\documentclass[12pt]{article}
\usepackage{amssymb,amsfonts,amsmath,amsthm, breqn, color, float, mathtools, booktabs, hvfloat, url}
\usepackage{adjustbox}
\usepackage{caption}
\newcommand{\one}[1]{\mbox {\bf 1}_{\{#1\}}}
\newcommand{\odin}{\mbox {\bf 1}}

\newcommand{\witi}{\widetilde}

\newcommand{\zz}{{\mathbb Z}}
\newcommand{\nn}{{\mathbb N}}
\newcommand{\qq}{{\mathbb Q}}
\newcommand{\rr}{{\mathbb R}}
\newcommand{\cc}{{\mathbb C}}
\newcommand{\cald}{{\mathcal D}}
\newcommand{\call}{{\mathcal L}}
\newcommand{\caln}{{\mathcal N}}
\newcommand{\calr}{{\mathcal R}}

\newcommand{\beq}{\begin{eqnarray*}}
	\newcommand{\feq}{\end{eqnarray*}}
\newcommand{\beqn}{\begin{eqnarray}}
\newcommand{\feqn}{\end{eqnarray}}

\newtheorem{theorem}{Theorem}
\makeatletter \@addtoreset{theorem}{section}\makeatother

\newtheorem*{theorem*}{Theorem}
\newtheorem*{conj*}{Conjecture}

\newtheorem{corollary}[theorem]{Corollary}
\newtheorem{remark}[theorem]{Remark}

\usepackage{algorithm}
\usepackage{graphicx}
\usepackage[noend]{algpseudocode}
\makeatletter
\def\BState{\State\hskip-\ALG@thistlm}
\makeatother

\DeclareCaptionLabelFormat{noname}{#2}
\newlength\myindent
\title{On ballistic deposition process on a strip}

\date{March 29, 2019; Revised June 10, 2019}
	\author{Toufik~Mansour\thanks{ Department of Mathematics, University of Haifa, 199 Abba Khoushy Ave, 3498838 Haifa, Israel;
			\newline e-mail: tmansour@univ.haifa.ac.il}
		\and
		Reza~Rastegar\thanks{Occidental Petroleum Corporation, Houston, TX 77046 and Departments of Mathematics
			and Petroleum Engineering, University of Tulsa, OK 74104, USA - Adjunct Professor; e-mail:  reza\_rastegar2@oxy.com}
		\and
		Alexander~Roitershtein \thanks{Department of Statistics, Texas A\&M University, College Station, TX 77843, USA;
			\newline e-mail: alexander@stat.tamu.edu}
	}

\begin{document}
	\maketitle
	
	\begin{abstract}
	We revisit the model of the ballistic deposition studied in \cite{bdeposition} and prove several combinatorial properties of the random tree structure formed by the underlying stochastic process. Our results include limit theorems for the number of roots and the empirical average of the distance between two successive roots of the underlying tree-like structure as well as certain intricate moments calculations.
	\end{abstract}
	\section{Introduction}
	\label{intro}
	Packing models arise in a variety of applied fields, including microscopic processes in physics, chemistry, and biology,
	and macroscopic ecological and sociological systems. One of the first proposed classes of packing processes are random sequential adsorption (RSA) models describing a process of deposition of thin disks (segments) placed at random one after another on a surface. When an attempt to deposit a new segment would result in an overlap with previously deposited one, this attempt is rejected. In statistical mechanics and biology models of this type are fundamental to the description of the irreversible deposition of macro-molecules, colloidal particles, viruses, polymer particles, and bacteria onto a surface. The model goes back to \cite{page1, renyi}, see, for instance, \cite{review, jim, penrose2, packingbook, talbot1} for a review of classical results and \cite{baule1, baule, shape, cubes, revisited, theater} and references therein for some very recent progress.
	\par
	The RSA packing model is generalized to the ballistic deposition (BD) processes where, in contrast to the RSA model, the segments are ``thick" and they do not get rejected but stick to the first point of contact, which might be either the surface or other segments \cite{book1, book, penrose2, kpz5, RSABD}. Thus the shape formed by the deposited particles not only expands on the surface but also grows vertically as a complex multilayered conglomerate. Similarly to the RSA, there is a vast literature concerned with various versions of the basic deposition processes on continuum and lattice substrates, most of it is a numerical simulation study. The BD models date back to \cite{vold} and \cite{suth}, where a variation of the model was proposed to describe sedimentation and aggregation in colloids. Models of this type have been applied to study formation, morphology, and surface roughness  of sedimentary rocks \cite{rocks} and  thin films \cite{films, films1}.
	\par
	We remark that a random deposition model, motivated by a cooperative sequential adsorption (CSA) \cite{jim, csa} rather than RSA,  has been recently considered in \cite{ecsa1, ecsa, stas}, see also a related ballistic deposition model proposed in \cite{family-amar}. Arguably, one of the most fascinating features of the BD models is that they are believed to belong to the  Kardar-Parisi-Zhang (KPZ) universality class \cite{kpz3, kpzt, kpz1, kpz4, kpz, anoma}, see also \cite{kartha, relaxation, kpz7, kpz5, bond} and references therein for some recent work in this direction for various types of BP models. For a general class of BD models \cite{penrose1, penrose2, penrose3, timo} established the existence of an asymptotic growth speed, thermodynamic limits, and asymptotically Gaussian fluctuations for the height and surface width of the random interface formed by the deposited particles. The main difficulty in the analysis of BD models is that local interactions of a deposited particle within a neighborhood of its projected location on the surface propagate into long-range spatial correlations and non-Markovian evolution of the model \cite{corwin}. The $(1+1)$-dimensional deposition process we consider in this paper was studied in \cite{bdeposition} as an analytically tractable variation of the diffusion limited aggregation model (DLA). For a compact review of DLA models in mathematical literature, we refer the reader to the recent article \cite{dlp}.
	\par
	We next describe the model that we are concerned with in this paper. For $K\in \zz,$ let $[K]$ denote $\{1, \cdots, K\}$ if $K\in\nn,$ and an empty set otherwise. Let $\nn_0=\nn \cup \{0\}.$ Informally speaking, we consider the $x$-axis in an $\rr^2$ plane, at each instance of time $n\in\nn_0$ we choose one site $k$ on the lattice substrate $[K],$ independently of the history and uniformly over $[K],$ and drop a solid rectangular particle of length $1$ and height $1$ vertically from above, with its center aiming at $k$. The particle will instantly fall down and stops upon touching the axis or a particle previously deposited within the neighbour set $I(k)$ which is defined as follows:
	\beqn
	\label{ik}
	I(k)=\left\{
	\begin{array}{ll}
		\{K,1,2\} & \quad \text{if} \quad k=1,\\
		\{k-1,k,k+1\} & \quad \text{if} \quad 2\leq k\leq K-1,\\
		\{K-1,K, 1\} & \quad \text{if} \quad k=K.
	\end{array}
	\right.
	\feqn
	See Figure \ref{fig:dpcartoon} for a graphical example.
	\begin{figure*}[t!]
		\centering
		\includegraphics[width=.50\textwidth]{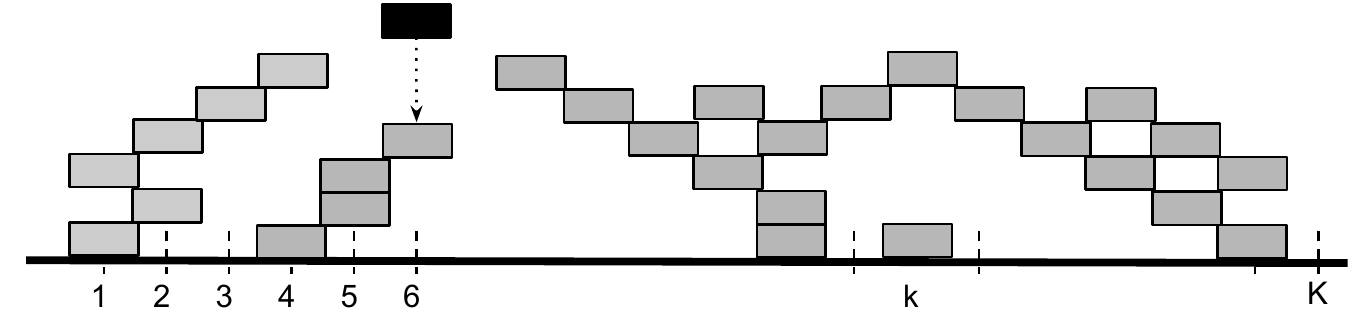}
		\caption{An instance  of deposition process at time $n=30.$ $H_{30}(4)=6.$}
		\label{fig:dpcartoon}
	\end{figure*}
	\par
	More precisely, let $(X_n)_{n\in \nn_0}$ be a sequence of i.\,i.\,d. random variables sampled uniformly from $[K].$ Let $(H_n)_{n \in \nn_0}$ be a sequence of random functions $H_n:[K]\to \nn_0$ representing the height of the deposited structure at each location at time $n$.
	Formally, set
	\beqn
	\label{hok}
	H_0(k) = 0 \qquad \forall\,k\in[K],
	\feqn
	and consider a Markov chain $H_n=\bigl(H_n(1),\ldots,H_n(K)\bigr)$ of vectors in the state space $\nn_0^K,$ defined recursively as follows:
	\beqn
	\label{hk}
	H_n(k) =\left\{
	\begin{array}{ll}
		H_{n-1}(k) & \quad \text{if} \quad X_n \neq k \\
		\max_{j\in I(k)} H_{n-1}(j) +  1 &\quad \text{otherwise},
	\end{array}\right.
	\feqn
	where $k\in[K]$ and the sets $I(k)$ are introduced in \eqref{ik}. We refer to the Markov chain $(X_n,H_n)_{n\in\nn_0}$ as a \emph{ballistic deposition on a strip}. Note that the cyclic rule \eqref{ik} effectively turns $[K]$ into a $K$-dimensional discrete torus in which $K$ is identified with zero.\\
	\begin{figure*}[hb!]
		\centering
		\begin{tabular}{ccc}
			\includegraphics[width=.30\textwidth]{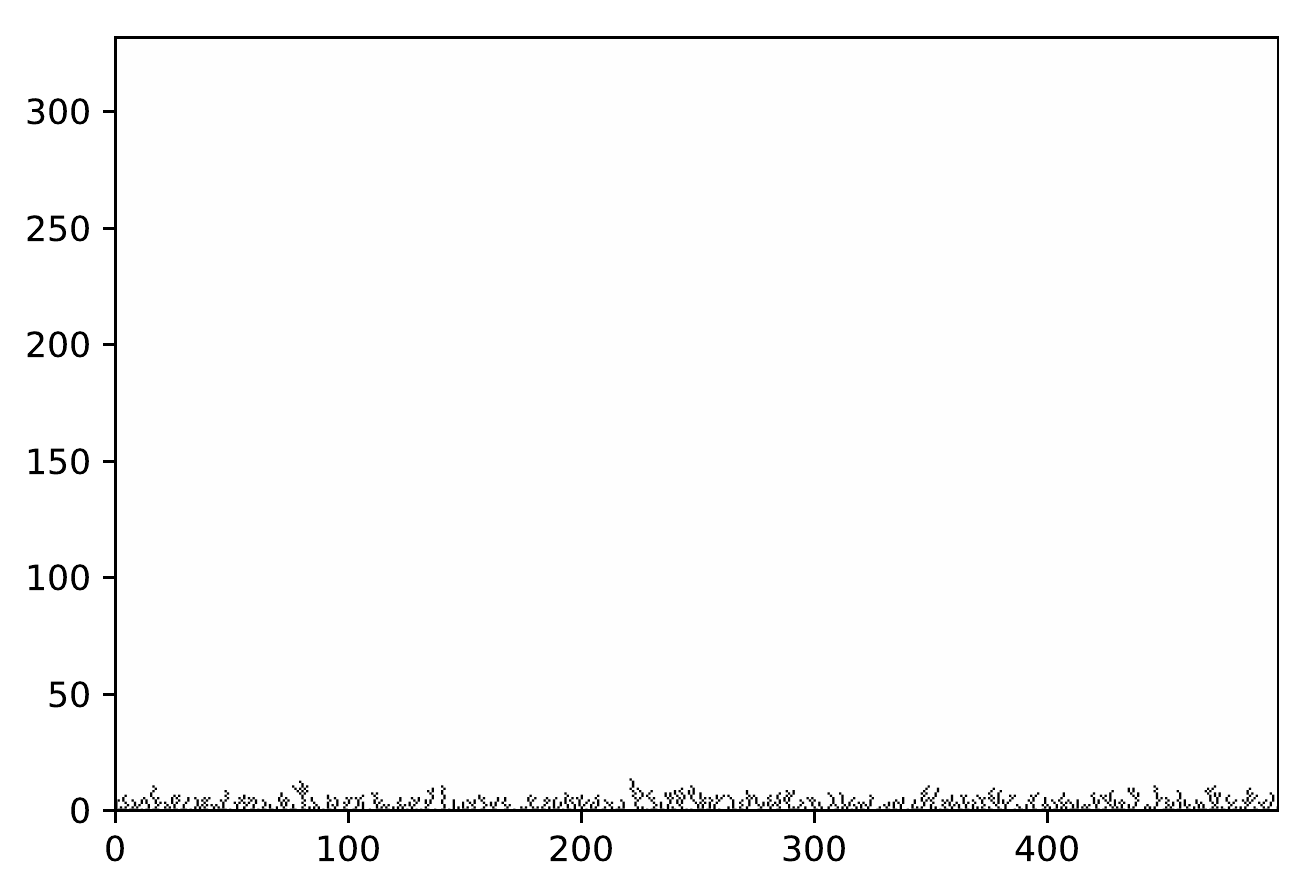}  &\quad $\mbox{}$ \quad  &\includegraphics[width=.30\textwidth]{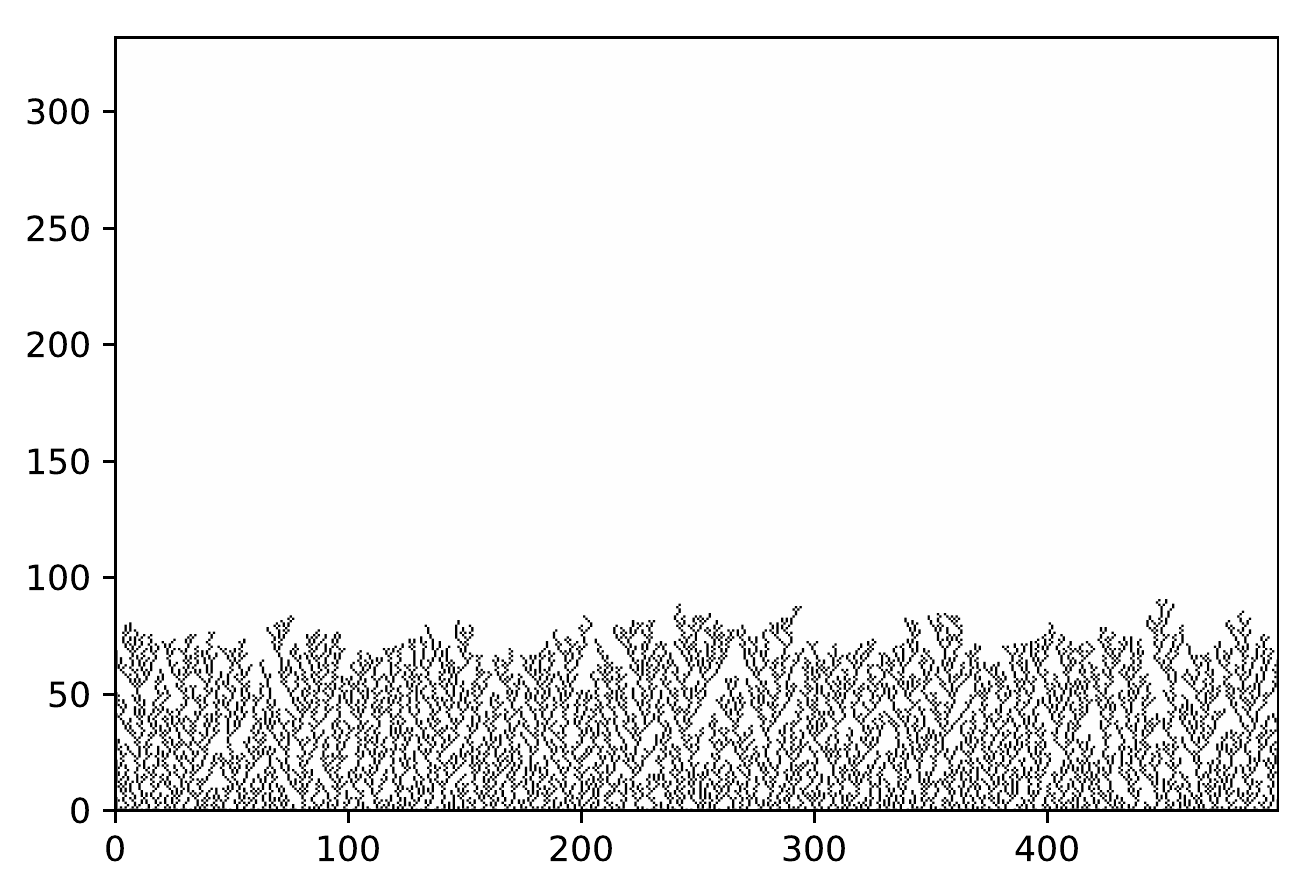}\\
			(a) $1000$ Iterations & \quad $\mbox{}$ \quad  & (b) $10000$ Iterations \\[6pt]
			\includegraphics[width=.30\textwidth]{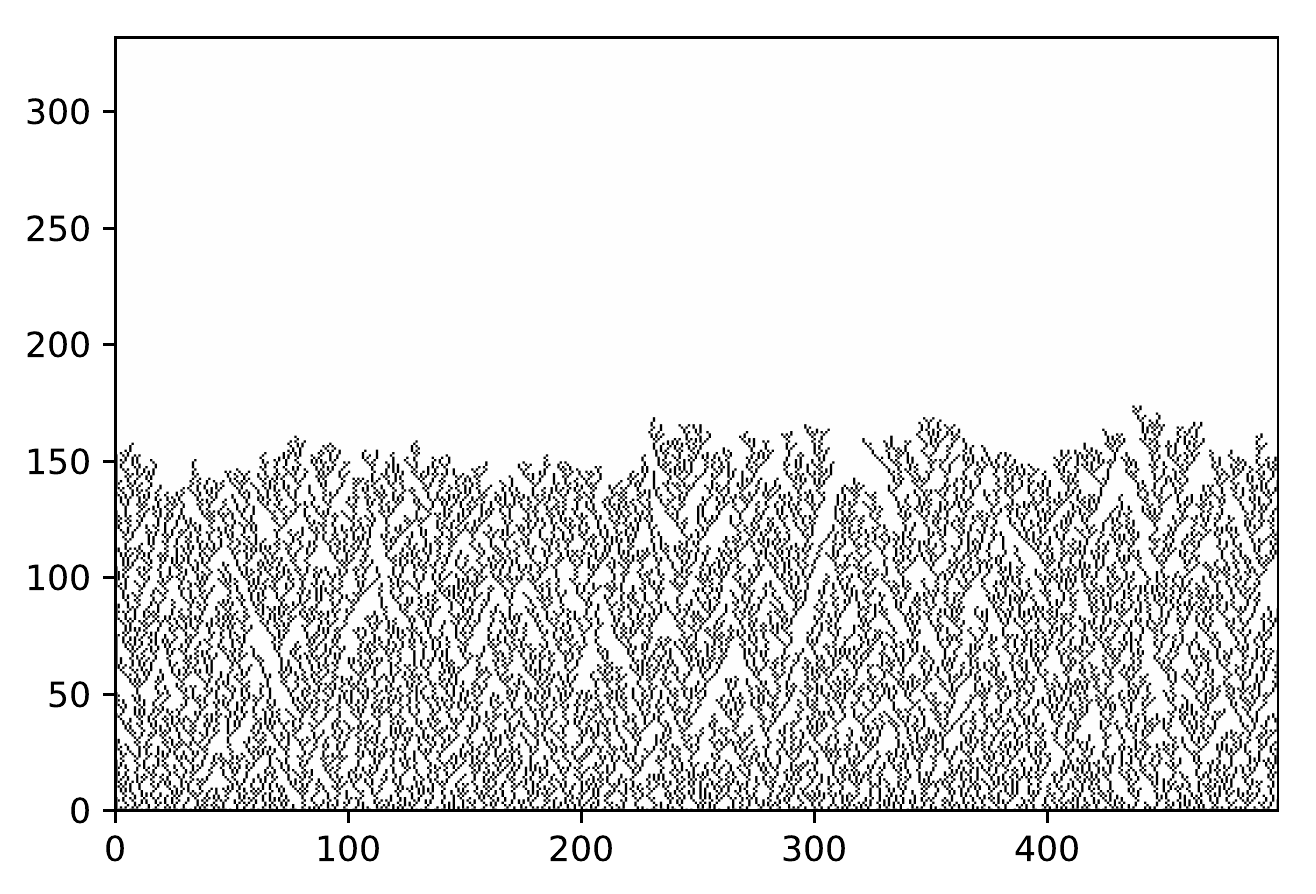} &\quad $\mbox{}$ \quad  & \includegraphics[width=0.30\textwidth]{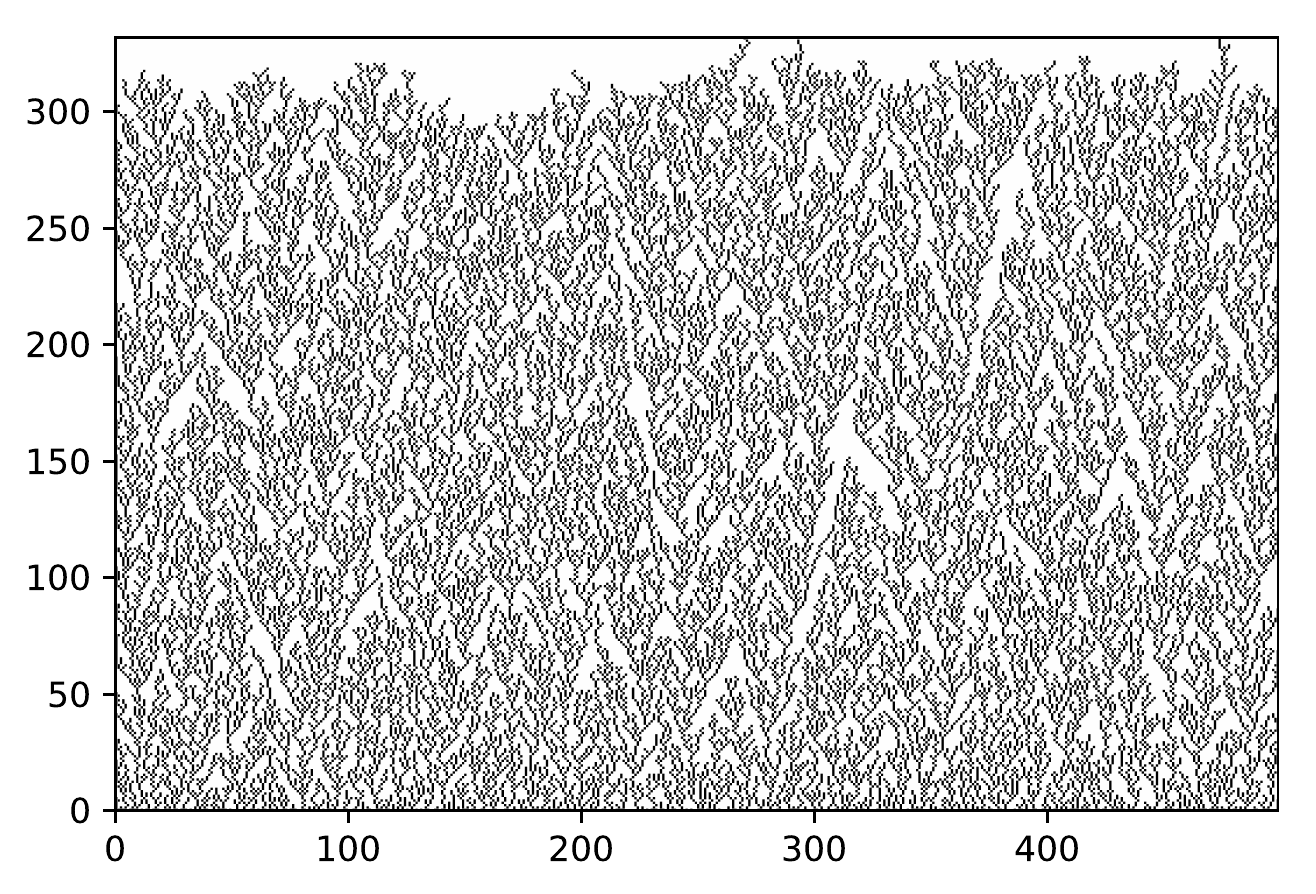} \\
			(c) $20000$ Iterations &\quad $\mbox{}$ \quad  & (d) $40000$ Iterations \\
			&\qquad $\mbox{}$ \qquad  &\\
		\end{tabular}
		\caption{Simulation of the process with $K=500$ for $n=40,000$ iterations.}
		\label{fig:simul40000}
	\end{figure*}
	\par
	Figure \ref{fig:simul40000} shows the outcome of 40,000 iterations of this process simulated numerically for a strip with $K=500.$ A random number of tree-like structures (connected components) grow and merge through the process. We refer to these structures as \emph{trees} even though they are not trees in a classical sense. Through several coupling arguments, a lower and an upper bounds for $\max_{k\in [K]}H_n(k)$ are calculated in \cite{bdeposition}. Our simulations warrant
	\begin{conj*}
		\label{c}
		With probability one, for all $j\in\nn,$
		\beq
		\lim_{n\to\infty} \frac{H_n(j)}{n}=\lim_{n\to\infty} \Bigl(\max_{k\in [K]} \frac{H_n(k)}{n}\Bigr)\sim \frac{4}{K},
		\feq
		where the notation $a_K\sim b_K$ stands for $\lim_{K\to\infty} \frac{a_K}{b_K}=1.$
	\end{conj*}
	\par
	The goal of this paper is to study the configuration of particles deposited directly on the surface, i.\,e. \emph{roots} of the trees formed by the deposed particles. More precisely, we focus on the probability distributions of the number of the particles eventually located on the surface (Section~\ref{sec:roots}) and distances between them (Section~\ref{sec:distance}). This information can serve as a basis for future investigation of the process as an evolving in time conglomerate of trees. Though problems of this type were intensively investigated for RSA models, to the best of our knowledge there is no previous work addressing the issue in the context of ballistic depositions. In  terms of the principle object of study (but not the methods), the closest to our line of inquiry work that we are aware of is \cite{analdp}, where the formation of the first layer is studied for a significantly different ``ballistic deposition with restructuring" model. A monolayer ballistic deposition model on a $2$-dimensional continuum is considered
in \cite{stat}.
	\par
	The main results of this paper are stated in Theorem~\ref{cltn} (exact moments, weak law of large numbers, and a CLT for the number of roots), Theorem~\ref{dltn} (limit theorem for the empirical average of a gap between two successive roots), and Theorems~\ref{thm:D1K} and~\ref{thm:D1K1} (exact moments for the distribution of the number of gaps of a given length between two successive  roots) together with weak laws of large numbers implied by the latter (Corollaries~\ref{cln} and~\ref{cln1}). See also Remark~\ref{rem1} concerning large deviation estimates, Berry-Essen Bounds, and a local CLT accompanying the CLT obtained in Theorem~\ref{cltn} as well as a conjecture regarding a CLT for the number of gaps of a given length and their joint distribution stated at Section~\ref{clti}.
\par
Our proofs rely on the analysis of recursive equations for underlying generating functions. Most of our moment calculations are exact rather than asymptotic. Some of the calculations are computationally intensive, we believe that the method developed in Section~\ref{sec:distance} in order to handle the computational complexity may be of independent interest.
	\section{Number of roots}
	\label{sec:roots}
	In this section we study the distribution of the number of particles located directly on the surface. We  refer to particles located on the surface as roots. The set of locations of the roots at time $n$ is defined as
	\beq
	\calr^{[K]}_n = \bigl\{ k\in [K] : H_i(k) = 1 \ \text{for some}\ i\leq n  \bigr\}.
	\feq
	We denote by $\calr^{[K]}$ the set of all roots eventually formed by the deposition process.
	That is,
	\beqn
	\label{rk}
	\calr^{[K]} = \lim_{n\to\infty} \calr^{[K]}_n=\bigl\{ k\in [K] : H_i(k) = 1  \ \text{at some time } i\in \nn \bigr\}.
	\feqn
	The convergence of the sequence $\calr^{[K]}_n$ to $\calr^{[K]}$ is granted because the sequence is formed by non-decreasing subsets of a finite set $[K].$
	\par
	In this paper we are concerned with $R^{[K]} = \text{Card}(\calr^{[K]}).$ The evolution of the sequence $R_n^{[K]} = \text{Card}(\calr_n^{[K]})$ will be studied by the authors in more detail elsewhere. Figure \ref{fig:Rksim} shows the empirical distribution of $R^{[K]}$ obtained in simulations for $K=100, 300$, $500,$ and $1500.$
	\begin{figure*}[ht!]
		\centering
		\begin{tabular}{cccc}
			\includegraphics[width=.20\textwidth]{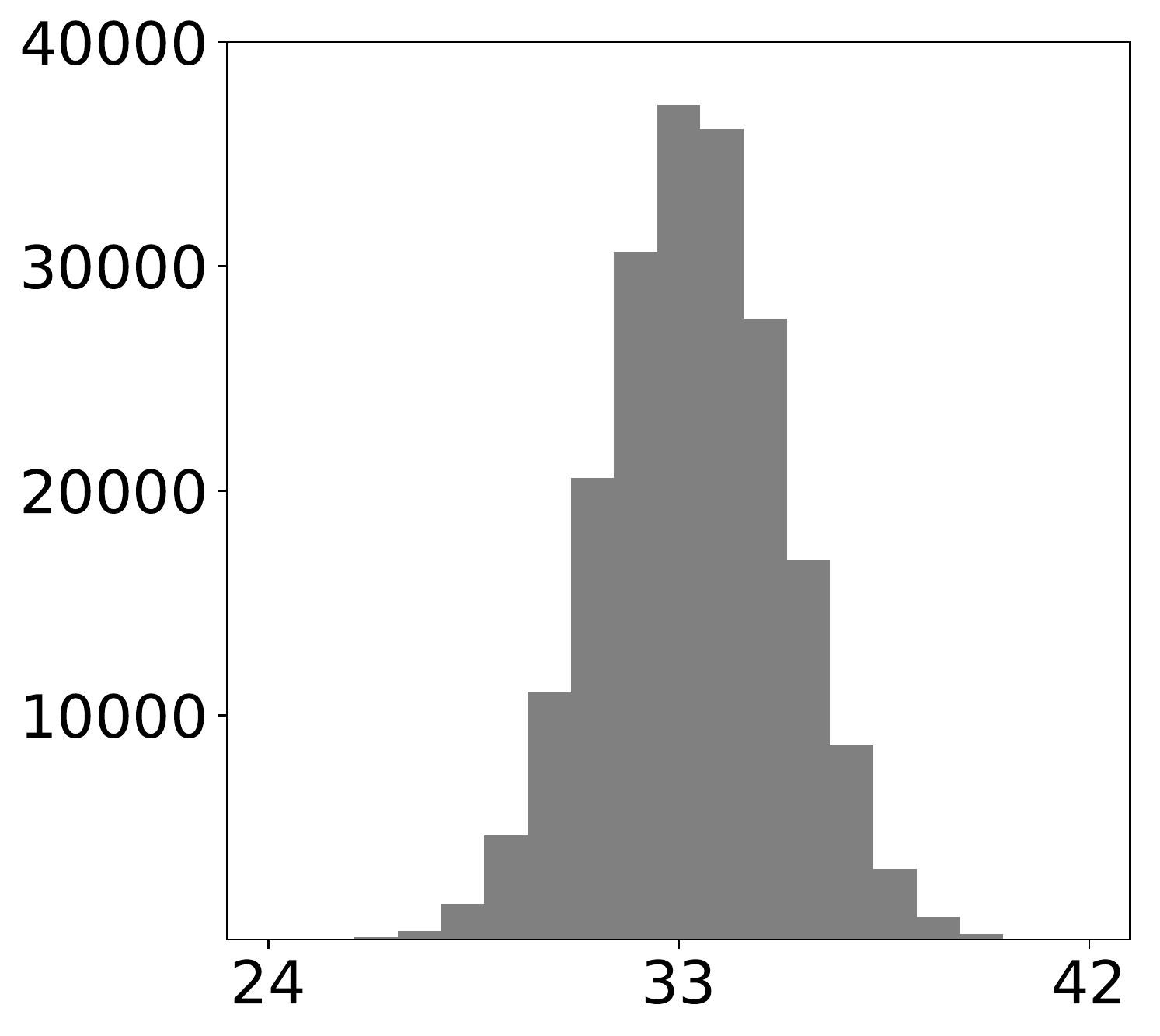}  &
			\includegraphics[width=.20\textwidth]{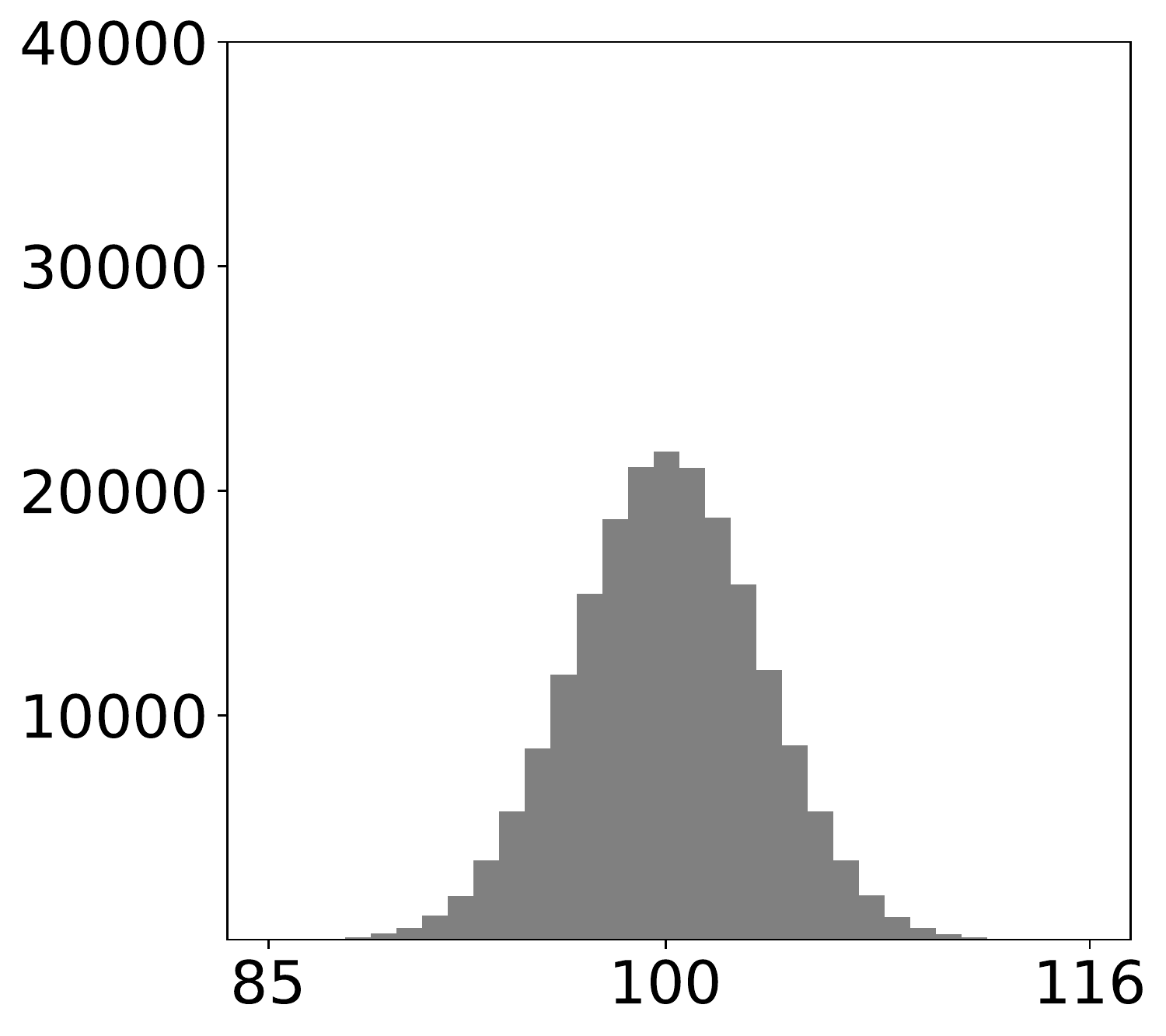}  &
			\includegraphics[width=.20\textwidth]{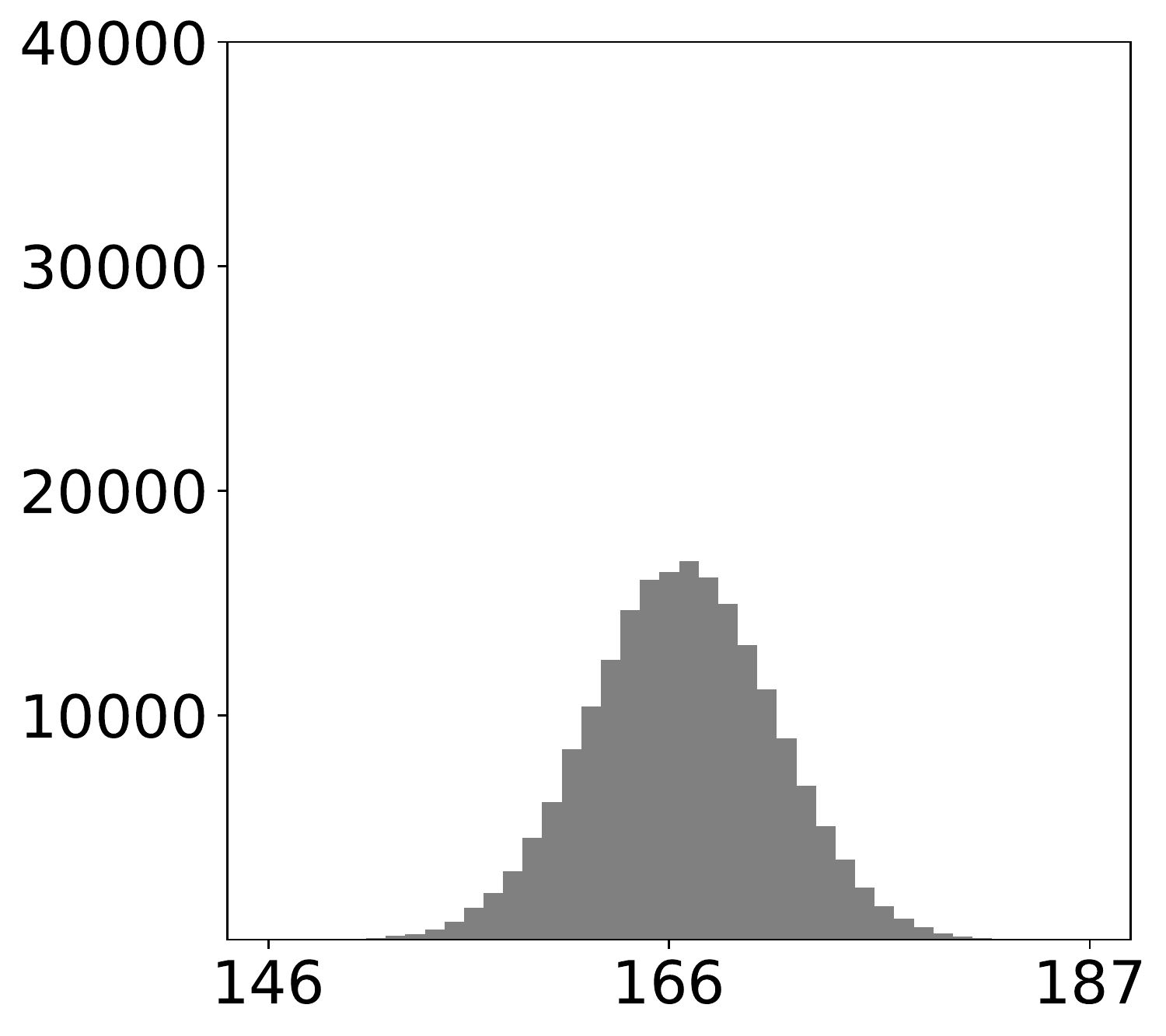}  &
			\includegraphics[width=.20\textwidth]{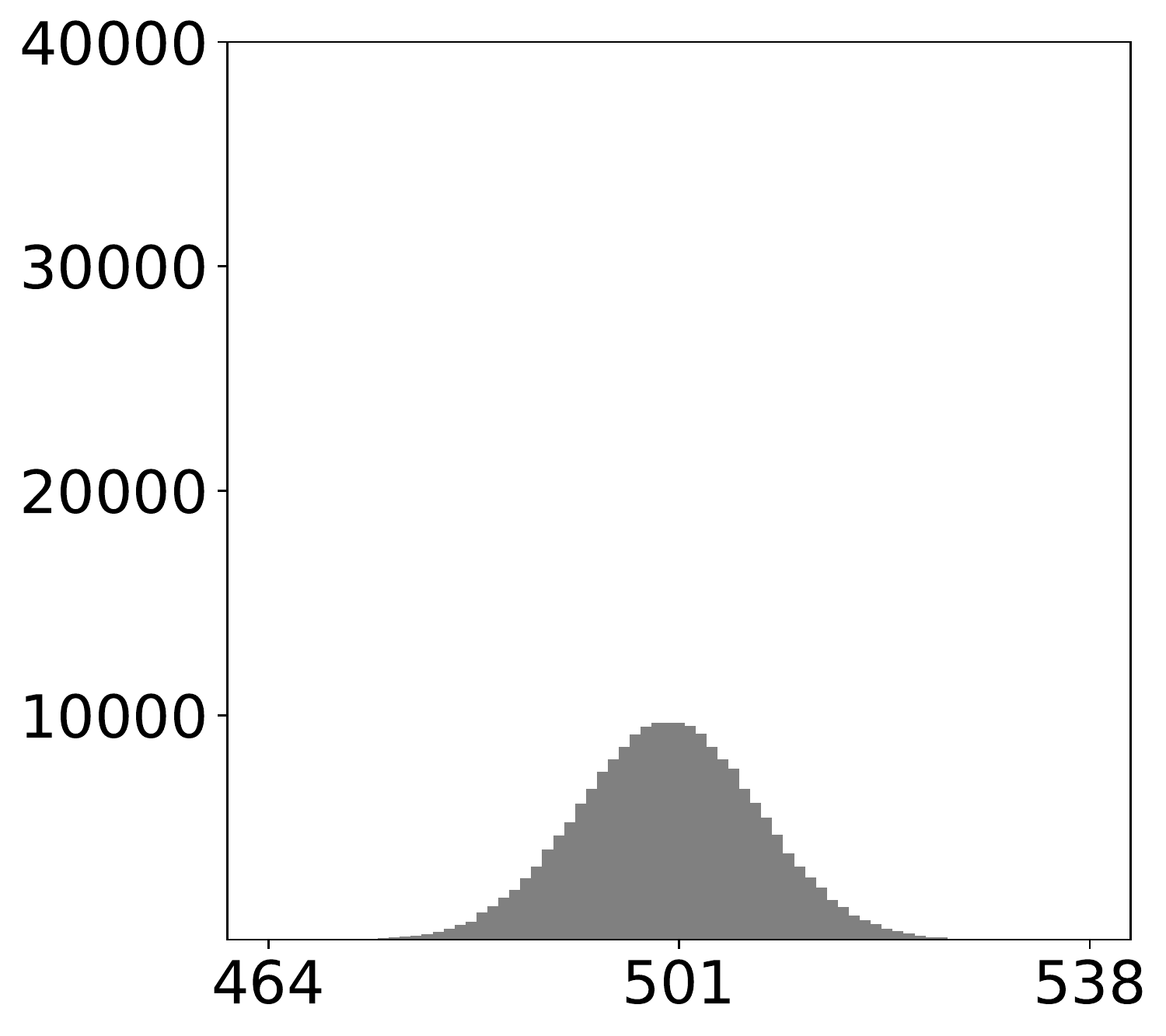} \\
			(a) $K=100$ & (b) $K=300$ & (c) $K=500$ & (d) $K=1500$ \\ [6pt]
		\end{tabular}
		\caption{Empirical distribution of the limit number of roots, $R^{[K]}$, for various values of $K.$ Each histogram is based on a simulation of $200,000$ runs.}
		\label{fig:Rksim}
	\end{figure*}	
	\par
	The simulations suggest that the random variable $R^{[K]}$ is asymptotically normal as $K$ approaches infinity. The corresponding formal statement is the content of the following theorem. Heuristically, one can expect a CLT to hold because $R^{[K]}=\sum_{i=1}^K \eta_i^{(K)},$ where $\eta_i^{(K)}=\odin_{\{\text{$i$ is a root}\}}$ is the indicator of the event $\{i\in \calr^{[K]}\}$, and the random variables $\bigl(\eta_i^{(K)}\bigr)_{K\in\nn,i\in [
	K]}$ form a uniformly mixing triangular array, the middle bulk of which is stationary asymptotically. Here and henceforth, $\odin_A$ stands for the indicator function of the event $A.$ For generic examples of limit theorems for triangular arrays of this type see \cite{mixing3,mixing, mixing1}. Though our proof doesn't use any of limit theorems for mixing arrays, and our asymptotic analysis of the characteristic function of $R^{[K]}$ relies on recursions obtained through the use of underlying combinatorial structure rather than on a direct exploiting of mixing properties, the weak dependence and approximate stationarity of the indicators seems to be a good intuitive way to understand the asymptotic normality of a properly normalized sequence $R^{[K]}$ (cf. \cite{clt3}).
\par
For a random variable $X,$ we denote its mean and variance by, respectively, $E(X)$ and $\sigma^2(X).$ In addition, we denote by $\call(X)$ its probability distribution. We use the notation $\lim_{K\to\infty} \call(X_K)=Y$ to indicate the convergence in distribution of a sequence of random variables $X_K,$ $K\in\nn,$ to a random variable $Y,$ as $K$ tends to infinity.
	\begin{theorem}
		\label{cltn}
		The following holds true for $R^{[K]}$:
		\begin{itemize}
			\item [(i)] $E(R^{[K]}) = \frac{K}{3}$ for all $K\geq 3.$
			\item [(ii)] $\sigma^2(R^{[3]})=0,$ $\sigma^2(R^{[4]})=\frac{1}{9},$ $\sigma^2(R^{[5]})=\frac{2}{9},$ $\sigma^2(R^{[6]})=\frac{4}{15},$  and
			\beq
			\sigma^2(R^{[K]}) = \frac{2K}{45}\qquad \forall\,K\geq 7.
			\feq
			\item [(iii)] Let $\witi {R}^{[K]} = \frac{R^{[K]} - E(R^{[K]})}{\sigma(R^{[K]})}.$
			Then
			\beq
			\lim_{K\to \infty} \call(\witi {R}^{[K]}) = \caln(0,1),
			\feq
			where $\caln(0,1)$ is a standard normal random variable.
		\end{itemize}
	\end{theorem}
	\begin{remark}
		\label{rem1}
		In order to prove the limit theorem for $R^{[K]}$ we employ a version of Hwang's general CLT (quasi-power theorem) \cite{Hwang1}. We refer an interested reader to Section~IX in \cite{analcombin} for a comprehensive account of the quasi-power theorem and its history. In fact, general results  available in \cite{Hwang1, Hwang2} can be used to obtain more detailed information about the limiting behavior of $R^{[K]}$ than it is given in part (ii) of Theorem~\ref{cltn}. More specifically, it is not hard to verify that our key estimate given in \eqref{Rlocal} implies that $R^{[K]}$ satisfies the conditions of both Theorem~1 in \cite{Hwang1} and Theorem~1 in \cite{Hwang2}. An application of these result yields large deviation estimates, local central limit theorem, and Berry-Essen type estimates for the distribution of $R^{[K]}.$ In particular, it turns out that the rate of convergence to the normal distribution in part (iii) of Theorem~\ref{cltn} is of order $n^{-1/2}.$ We omit the details, and instead refer the reader to the statement of the results in \cite{Hwang1, Hwang2}. Hwang's theory produces the asymptotic form of $E(R^{[K]})$ and $\sigma^2(R^{[K]})$ as a byproduct. Therefore, the proof of the limit theorem in part (iii) is in fact independent from the computation in parts (i) and (ii). The latter are included because they give the exact values of the expectation and variance, and hence may be of independent interest.
	\end{remark}
	\begin{proof}[Proof of Theorem~\ref{cltn}]
		\item [(i)] Consider a slight modification of the underlying process $H_n$ which is formally obtained by replacing the definition of $I(k)$ in \eqref{ik} with
		\beqn
		\label{ik1}
		I(k)=\left\{
		\begin{array}{ll}
			\{0,1,2\} & \quad \text{if} \quad k=1,\\
			\{k-1,k,k+1\} & \quad \text{if} \quad 2\leq k\leq K-1,\\
			\{K-1,K, K+1\} & \quad \text{if} \quad k=K.
		\end{array}
		\right.
		\feqn
	and the initial condition in \eqref{hok} by the following one:
		\beqn
		\label{hok1}
		H_0(0)=H_0(K+1)=1\qquad \mbox{and}\qquad H_0(k) = 0 \quad \mbox{\rm for}~k\in[K].
		\feqn

		Thus the ballistic deposition in the auxiliary process occurs on the same lattice substrate $[K]$ and according to the same rule \eqref{hk} as in the original one, with the only two exceptions being that (i) two particles are placed before the process starts at the external boundary $\{0,K+1\},$ and (ii) by virtue of \eqref{ik1}, the surface represented by the interval $[K]$ is not anymore cyclic, cf. \eqref{ik}. Note that according to our definition, similarly to the original process, particles in the auxiliary one are never deposited outside of the interval $[K]$ after time zero. Informally, the auxiliary process on the substrate $[K]$ (ignoring the initial particles at $0$ and $K+1$) coincides with the original cyclic one, observed on the substrate $[K+1]$ \emph{after} the arrival of the first particle (ignoring the first particle).
\par
Let $R_K$ be the limiting number of roots, i.\,e. the analogue of $R^{[K]}$ in \eqref{rk}, in the auxiliary process. Observe that (now counting two initial 
particles in the auxiliary process and the first particle in the original one) 
		\beq
		\call\bigl(R^{[K]}\bigr) = \call\bigl(R_{K-1} + 1\bigr),
		\feq
		and hence it suffices to analyze $R_K.$ The first-step decomposition of $R_K$ translates into the following
		distributional recursion:
		\beqn \label{tauk_recursion}
		\call(R_K) = \call \left(\odin_{\{Y_K = 1\}} R^{(1)}_{K-1} + \sum_{j=2}^{K-1} \odin_{\{Y_K = j\}} (R^{(1,j)}_{j-1} + R^{(2,j)}_{K-j} + 1) + \odin_{\{Y_K = K\}} R^{(K)}_{K-1}\right),
		\feqn
		where
		\begin{itemize}
			\item [-] $R_0 = R_1= R_2 = 0,$
			\item [-] $Y_K$ is the location on the surface of the first particle,
			\item [-] $\call(R_{k}) = \call(R^{(1)}_{k}) = \call(R^{(1,j)}_{k}) = \call(R^{(2,j)}_{k}) = \call(R^{(K)}_{k})$ for all $k\in [K],$
			\item [-] $R_{k}$, $R^{(1)}_{k}$, $R^{(1,j)}_{k}$, $R^{(2,j)}_{m}$, $R^{(K)}_{k}$ and $Y_K$ are independent of each other for all values of the arguments $k,j,m,$ and $K.$
		\end{itemize}
		Let $L_K(z) = E(z^{R_K}),$ $z\in\cc,$ be the generating function of $R_K$ with the domain in the complex plane.
		Note that $L_0(z) = L_1(z) = L_2(z) = 1$ and $L_3(z) = \frac{1}{3}(2+z).$ Since $R_K\leq K,$ the generating function is well defined and analytic in $\cc.$ In particular, due to the initial condition \eqref{hok1}, $L_K(0)=P(R_K=0)=\frac{2^{K-3}}{K!}.$ It follows from \eqref{tauk_recursion} that for $K\geq 3,$
		\beqn
		L_{K}(z) &=& E(z^{R_{K}}) = \frac{1}{K}\left( E(z^{R_{K-1}}) + E(z^{R_{K-1}})
		+ z \sum_{j=2}^{K-1}{ E(z^{R_{j-1}})E(z^{R_{K-j}})}\right)
		\nonumber \\
		&=& \frac{1}{K}\left( 2 L_{K-1}(z) + z \sum_{j=2}^{K-1}{L_{j-1}(z)L_{K-j}(z)}\right) \label{dkrecursion}.
		\feqn
		In order to calculate the first moment of $R_K,$ we take the derivative at $z=1$  on both sides of the identity in \eqref{dkrecursion}, and obtain
		\beq
		E(R_K) = \frac{1}{K} \left(2 E(R_{K-1}) + (K-2) + \sum_{j=2}^{K-1}\bigl[ E(R_{j-1}) + E(R_{K-j})\bigr] \right).
		\feq
		Therefore, since $E(R_0) = E(R_1) = E(R_2) = 0,$ for $k\geq 3$ we get
		\beq
		K E(R_K) = 2 \sum_{j=3}^{K-1} E(R_j) + (K-2).
		\feq
		Subtracting from this identity
		\beq
		(K-1) E(R_{K-1}) = 2 \sum_{j=3}^{K-2} E(R_j) + (K-3)
		\feq
		and solving the resulting first-order linear recursion
		\beq
		KE(R_K)=(K+1)E(R_{K-1})+1
		\feq
		with the boundary condition $E(R_2)=0,$ we obtain that for $K\geq 3,$
		\beqn \label{Etauk}
		E(R_{K}) = \frac{K-2}{3} \quad \text{and} \quad E(R^{[K]}) = E(R_{K-1}) + 1 = \frac{K}{3}.
		\feqn
		\item [(ii)]		Similarly, to calculate the second moment of $R_K$ and $R^{[K]},$ we take the second derivative at $z=1$ in both sides of \eqref{dkrecursion}, and obtain
		\beq
		K \{E(R^2_K)-E(R_K)\} &=& 2\{E(R^2_{K-1})-E(R_{K-1})\}+ 2\sum_{j=1}^{K-2}\{E(R_j^2)-E(R_j)\}
		\\
		&&
		\qquad
		+4\sum_{j=1}^{K-2}E(R_j)
		+ 2\sum_{j=1}^{K-2}E(R_j)E(R_{K-j-1}).
		\feq
		Therefore, for $K\geq 3,$
		\beqn
		\label{k1}
		K E(R^2_K)=
		2 \sum_{j=1}^{K-1} E(R_j^2) + 2\sum_{j=1}^{K-2}E(R_j)\{E(R_{K-j-1})+1\} +\frac{K^2-4K+6}{3}.
		\feqn
		It is easy to check directly that $E(R_3^2)=\frac{1}{3}$ and $E(R_4^2)=\frac{2}{3}.$  Then, using \eqref{k1} we get
		$E(R_5^2)=\frac{19}{15}.$  It follows from \eqref{k1} that for $K\geq 6,$
		\beqn
		\label{k}		
		K E(R^2_K)&=& 2 \sum_{j=1}^{K-1} E(R_j^2) + 2\sum_{j=3}^{K-3}\Bigl(\frac{j-2}{3}\cdot \frac{K-j}{3}\Bigr)+\frac{K^2-2K-2}{3}.
		\feqn
		In particular, $E(R_6^2)=\frac{188}{90}.$ For $K\geq 7,$ subtracting from \eqref{k} the identity
		\beq
		(K-1) E(R^2_{K-1}) = 2 \sum_{j=1}^{K-2} E(R_j^2) + 2\sum_{j=3}^{K-4}\Bigl(\frac{j-2}{3}\cdot\frac{K-j-1}{3}\Bigr)+\frac{K^2-4K+1}{3}
		\feq
		yields the first-order linear recursion
		\beq
		K E(R^2_K)=(K+1) E(R^2_{K-1})+\frac{K^2+K-9}{9}, \qquad K\geq 7,
		\feq
		with the boundary condition $E(R_6^2)=\frac{188}{90}.$ Let $E(R^2_K)=\alpha_k(K+1).$ Then $\alpha_6=\frac{188}{630},$ and for $K\geq 7$ we get
		\beq
		\alpha_K = \alpha_{K-1} + \frac{K^2+K-9}{9K(K+1)}=\alpha_{K-1} + \frac{1}{9}-\Bigl(\frac{1}{K}-\frac{1}{K+1}\Bigr).
		\feq
		Iterating and taking in account that $\alpha_6=\frac{188}{630},$ we obtain
		\beq
		\alpha_K =\frac{188}{630}+ \frac{K-6}{9}-\Bigl(\frac{1}{7}-\frac{1}{K+1}\Bigr),
		\feq
		and hence
		\beq
		\sigma^2(R^{[K+1]})=\sigma^2(R_K)=\alpha_{K+1}^2(K+1)^2-\frac{(K-2)^2}{9}=\frac{2}{45}(K+1),
		\feq
		as desired.
		\item [(iii)] To show that the CLT holds for $R^{[K]}$ we will verify that the conditions of Hwang's quasi-power theorem hold for $L_k(z)$ (the version of this general combinatorial CLT given in Theorem~IX.8 of \cite{analcombin} will be sufficient for the purpose). Toward this end, consider the generating function
		\beq
		L(x, z) = \sum_{K=1}^\infty L_K(z) x^K
		\feq
		with the domain in $\cc^2.$ Since $R_K\leq K,$ the function is well defined at least for all  $(x,z)\in\cc^2$ such that $|x|<\min\{1,|z|^{-1}\}.$ We will be interested in the behavior of $L(x,z)$ in an open $\cc^2$-neighborhood of $(x,z)=(0,1)$ where $L(x,z)$ is well defined and analytic as a function of $x$ for each fixed $z.$
		\par
		Substituting \eqref{dkrecursion} into the definition of $L(x,z)$ gives
		\beq
		L(x, z) &=& x + x^2 + \sum_{K=3}^{\infty} \frac{1}{K}\left( 2 L_{K-1}(z) + z \sum_{j=2}^{K-1}{L_{j-1}(z)L_{K-j}(z)}\right)x^K.
		\feq
		Taking the partial derivative with respect to $x$ on both sides, we obtain the following inhomogeneous Ricatti equation:
		\beq
		\frac{\partial L(x, z)}{\partial x} &=&  1 + 2x + \sum_{K=3}^{\infty} \left( 2 L_{K-1}(z) + z \sum_{j=1}^{K-2}{L_j(z)L_{K-j-1}(z)}\right)x^{K-1}
		\\
		&=& 1 + 2 \sum_{K=1}^{\infty} L_{K}(z) x^K + z \left(\sum_{j=1}^{\infty} L_j(z) x^j \right)\left( \sum_{K = j + 2}^{\infty} {L_{K-j-1}(z) x^{K-j-1}}\right)
		\\
		&=& 1 + 2L(x,z) +  z L(x,z)^2
		\feq
		with the initial condition $L(0,z)=0.$ The solution $L(x,s)$ in a neighborhood of $(x,z)=(0,1)$ is given by \cite{hille}
		\beq
		L(x,z)= \left\{
		\begin{array}{ll}
			\frac{\tanh (x\sqrt{1-z})}{\sqrt{1-z}-\tanh (x\sqrt{1-z})} &\mbox{\rm if}~z\neq 1 ,\\
			\frac{x}{1-x}  &\mbox{\rm if}~z=1.
		\end{array}
		\right.
		\feq
		Since the singularity at $z=1$ is removable, we will simply write
		\beqn
		\label{Rxz_sol}
		L(x,z)=\frac{\tanh (x\sqrt{1-z})}{\sqrt{1-z}-\tanh (x\sqrt{1-z})}=\frac{\tan (x\sqrt{z-1})}{\sqrt{z-1}-\tan (x\sqrt{z-1})}=: \frac{Q(x,z)}{P(x,z)}.
		\feqn
		The poles of $L(x,z)$ for $z\neq 1$ are in the form $\rho_m(z)=\frac{\tan^{-1} (\sqrt{z-1})+\pi  m}{\sqrt{z-1}},$ $m\in\zz.$ Since
		\beq
		\lim_{z\to 1} \frac{\tan^{-1} (\sqrt{z-1})}{\sqrt{z-1}}=1,
		\feq
		there exists a complex punctured neighborhood $U$ of $1$  and a real number $r>1$ such that $|\rho_0(z)|<r$ and $|\rho_m(z)|>r$ for all $m\in\nn$ and $z\in U.$
		\par
		By the residue theorem, for $z\in U$ we have
		\beq
		\frac{1}{2\pi i}\int_{|x|=r}L(x,z) \frac{dx}{x^{K+1}} = \text{Res}_0(L(x,z)x^{-K-1}) + \text{Res}_{\rho_0(z)}(L(x,z)x^{-K-1}).
		\feq
		Since $\text{Res}_0(L(x,z)x^{-K-1})=L_K(z),$ we get
		\beqn
		\label{Rint}
		L_K(z) = -\text{Res}_{\rho_0(z)}(L(x,z)x^{-K-1}) + \frac{1}{2\pi i}\int_{|x|=r}L(x,z) \frac{dx}{x^{K+1}}.
		\feqn
		By virtue of \eqref{Rxz_sol},
		\beq
		\text{Res}_{\rho_0(z)}(L(x,z)x^{-K-1}) &=&  \lim_{x\to \rho_0(z)}(x-\rho_0(z))L(x,z)x^{-K-1}  \\
		&=& \lim_{x\to \rho_0(z)} \frac{Q(x,z)x^{-K-1}}{\frac{\partial P(x,z)}{\partial x}}=- \frac{\rho_0(z)^{-K-1}}{z},
		\feq
		where in order to compute the partial derivative in the denominator we used the fact that $\tan \bigl(\rho_0(z) \sqrt{z-1}\bigr)=\sqrt{z-1}.$ Since $L(x,z)$ is continuous and
		therefore bounded on the closure of $\{(x,z)\in\cc^2:|x|=r,z\in U\},$ we obtain from \eqref{Rint} that there exists a function $g$ on $U$ such that
		\beqn
		\label{Rlocal}
		L_K(z)  = \frac{\rho_0(z)^{-K-1}}{z} + g(z)r^{-K},\qquad \mbox{\rm where}\qquad \sup_{z\in U} |g(z)|<\infty.
		\feqn
		It is now a simple routine to verify that the conditions of the quasi-power theorem (see Theorem~IX.8 in \cite{analcombin}) are satisfied for $R^{[K]}.$
		The quasi-power theorem implies the CLT, and hence the proof of part (iii) of the theorem is complete.				
	\end{proof}	
	
	\section{Distance between two adjacent roots}
	\label{sec:distance}
	In this section we investigate the random vector compound of distances between adjacent roots in the set $\calr^{[K]}.$  Our main results here are stated in Theorems~\ref{dltn}, \ref{thm:D1K} and~\ref{thm:D1K} below, see also a simulation-supported conjecture which is formulated in Section~\ref{clti}.
\par
Let $r_1< \cdots<r_{R^{[K]}}$ be the ordered locations of the roots. Let
	\beq
	D_K=(D_{1,K}, D_{2,K}, \cdots, D_{K-1,K})
	\feq
	be a $(K-1)$-vector whose $i$-th component $D_{i,K}$ counts the number of pairs of consecutive roots with the distance between them equal to $i.$
That is,
	\beq
	D_{i, K} = \sum_{j=1}^{R^{[K]}-1} \one{r_{j+1}-r_j = i+1} +\one{ r_{R^{[K]}}-r_1 = K-i-1}.
	\feq
In what follows we focus on the study of $D_K.$ The section is divided into subsections as follows. Section~\ref{aver} is devoted to a discussion of the asymptotic behavior of certain ``mean-field" and empiric averages of $D_{i,K}.$ A central limit theorem for the empiric average is derived as a corollary to Theorem~\ref{cltn}, the result is stated in Theorem~\ref{dltn}. In Section~\ref{moments} we are concerned with first moments of the random variables $D_{i,K}.$ In order to compute the moments we implement an approach similar to the one we used in the previous section in order to analyze $R^{[K]}.$  The recursive equations that we obtain in this section are considerably more complex, and we believe that our method of solving them is of independent interest.
\par
The general method that we develop in Section~\ref{moments} is applied in Section~\ref{i1} to obtain an exact formula for the first and second moments of $D_{1,K}$ (Theorem~\ref{thm:D1K}) and in Section~\ref{ilarge} to obtain similar explicit results for moments of $D_{i,K}$ with arbitrary $K$ and $i\in [2,7]$ (Theorem~\ref{thm:D1K1}). In principle, the method allows to obtain similar results recursively for an arbitrary value of the parameter $i.$ Since both $E(D_{i,K})$ and $\sigma^2(D_{i,K})$ turns out to be linear in $K,$ a byproduct of the above theorems are weak laws of large numbers stated as Corollary~\ref{cln} (the case $i=1$) and Corollary~\ref{cln1} (the case $i\in [2,7]$). In Section~\ref{clti} we state a conjecture regarding the asymptotic normality of the vector $D_K.$ The result is supported by our simulations, but unfortunately we were unable to prove it analytically.
\par 
We remark that the results in Section~\ref{ilarge} are not as complete as the results for the number of roots in Section~\ref{sec:roots}. However, the exact computation of moments for probabilistic combinatorial structures is a rather common line of inquiry in combinatorics, in particular with the goal of proving limit theorems for dependent variables in mind, cf. \cite{janson,becite}. We therefore consider our Theorem~\ref{thm:D1K1} as a first step in the study of a challenging subject and hope that our proof method not only is of interest on its own in general, but also can be further developed to prove the results conjectured later in Section~\ref{clti}.
\subsection{Average gap between two successive roots}
\label{aver}
We begin with a simple observation regarding certain averages of the distance between two consecutive roots. For $i\in [K-1]$ let
\beq
q_i=\Bigl(\sum_{i=1}^{K-1}E( D_{i,K})\Bigr)^{-1}E(D_{i,K}) \qquad \mbox{\rm and}\qquad
p_i=\Bigl(\sum_{i=1}^{K-1}D_{i,K}\Bigr)^{-1} D_{i,K}.
\feq
Intuitively, $p_i$ and $q_i$ represent, respectively, the empirical and a ``mean-field" frequency of pairs of roots with distance $i$ between them. The fact
that $\sum_{i=1}^{K-1} D_{i,K}=R^{[K]}$ along with \eqref{Etauk} imply that
\beq
q_i=\frac{3E(D_{i,K})}{K} \qquad \mbox{\rm and}\qquad
p_i=\frac{D_{i,K}}{R^{[K]}}.
\feq
Observe that
	\beqn
\label{krd}
	K = R^{[K]} + \sum_{i=1}^{K-1} i D_{i,K}.
	\feqn
Therefore, by virtue of \eqref{Etauk}, for any $K\geq 3$ we have
\beqn
\label{mf}
\sum_{i=1}^{K-1} i E(D_{i,K}) = \frac{2K}{3},\qquad
\mbox{\rm or, equivalently,}
\qquad
 <D_{\centerdot ,K}>_q:=\sum_{i=1}^{K-1}iq_i=2,
\feqn
where $<D_{\centerdot ,K}>_q$ is the average distance between consecutive roots in the above ``mean-field" model. Similarly,
in view of \eqref{krd},
\beqn
\label{emp}
<D_{\centerdot ,K}>_p:=\sum_{i=1}^{K-1}ip_i=\frac{K}{R^{[K]}}-1.
\feqn
As a corollary to Theorem~\ref{cltn} we derive the following result for the empirical average.
\begin{theorem}
\label{dltn}
For $K\in\nn$ let $T^{[K]} =\sqrt{K}\bigl(<D_{\centerdot ,K}>_p-2\bigr).$
Then
\beq
\lim_{K\to \infty} \call(T^{[K]}) = \caln\Bigl(0,\frac{5}{18}\Bigr),
\feq
where $\caln\Bigl(0,\frac{5}{18}\Bigr)$ is a normal random variable with mean zero and variance $\frac{5}{18}.$
\end{theorem}
Figure~\ref{fig:meanfield} below show results of numerical simulations for $<D_{\centerdot ,K}>_p.$
\begin{figure*}[ht!]
		\centering
		\begin{tabular}{ccc}
			\includegraphics[width=.3\textwidth]{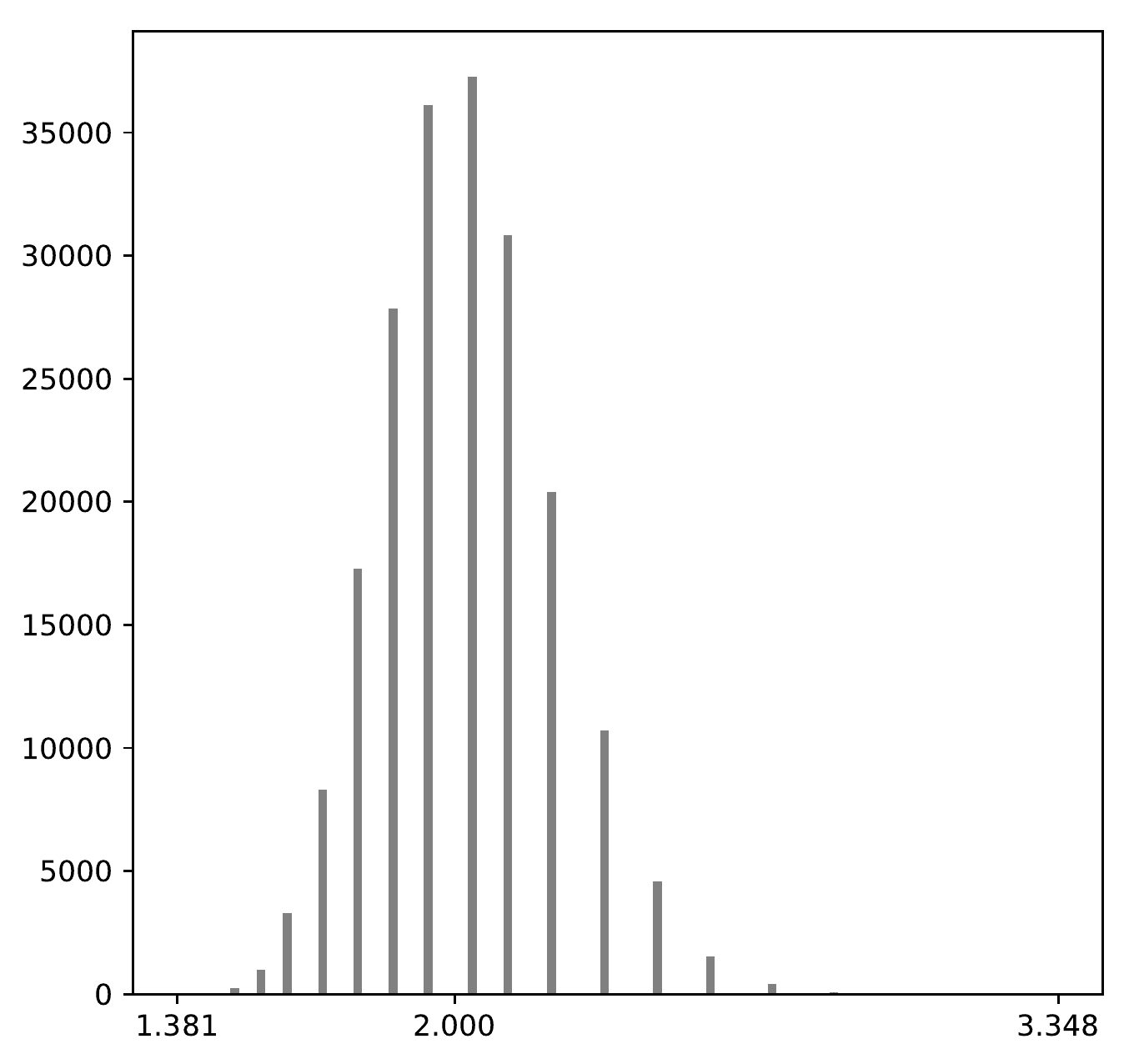}  &
			\includegraphics[width=.3\textwidth]{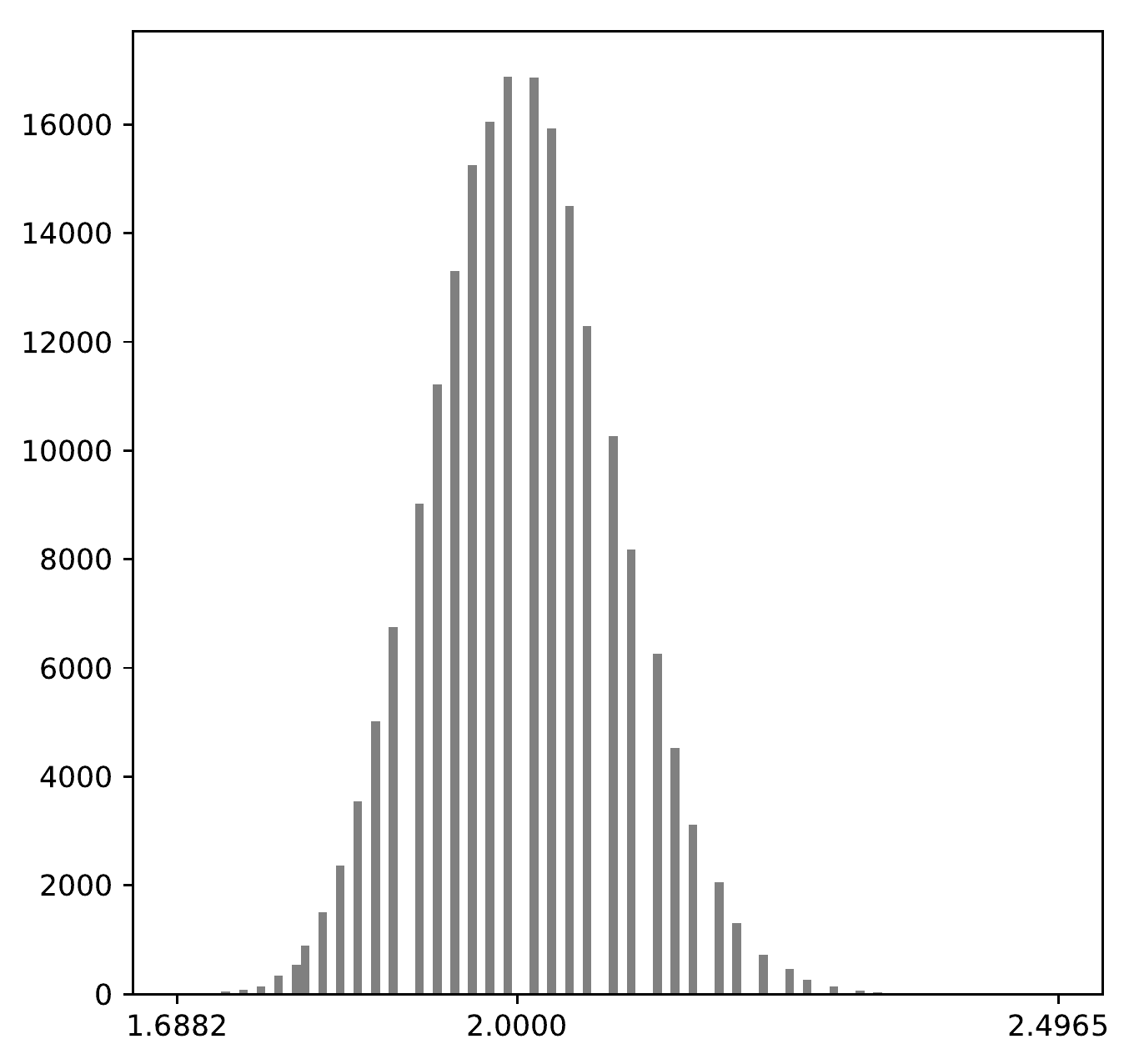}  &
			\includegraphics[width=.3\textwidth]{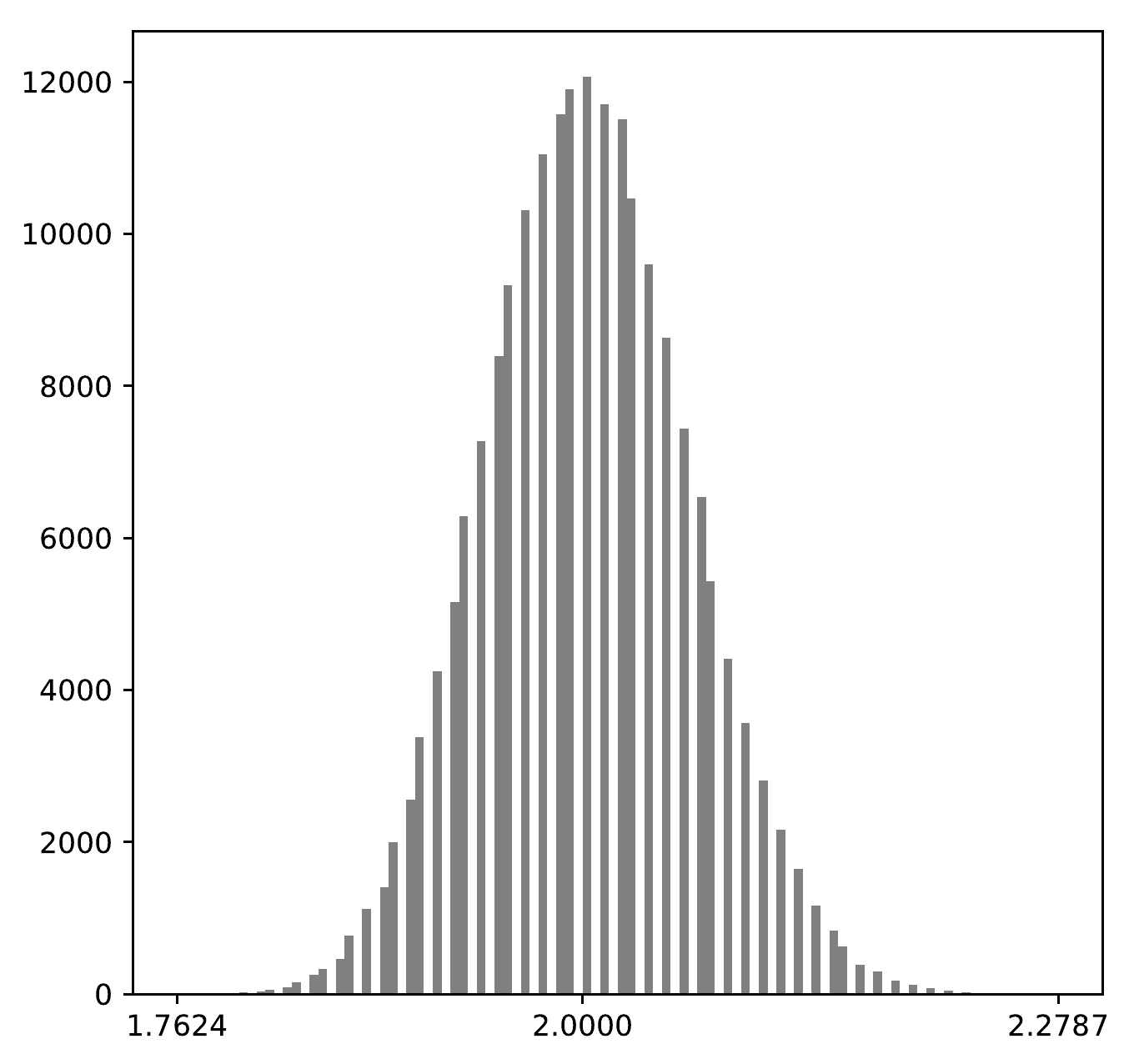}  \\
			(a) $K=100$ & (b) $K=500$ & (c) $K=1000$ \\ [6pt]
		\end{tabular}
		\caption{Empirical distribution of the empirical average $<D_{\centerdot ,K}>_p$ for various values of $K.$ Each histogram is based on a simulation of $200,000$ runs.}
		\label{fig:meanfield}
	\end{figure*}	

The proof of the theorem is a standard routine, we will only outline the argument for the sake of completeness.
Taking in account \eqref{emp}, write for an arbitrary $x\in\rr$ and any $K\in\nn$ large enough (specifically, we need $3\sqrt{K}+x>0$),
\beq
&& P\bigl(\sqrt{K}(<D_{\centerdot ,K}>_p-2)\leq x\bigr)=
P\Bigl(\frac{K}{R^{[K]}}-3\leq \frac{x} {\sqrt{K}}\Bigr)
=
P\Bigl(\frac{R^{[K]}}{K}-\frac{1}{3}\geq \frac{\sqrt{K}}{3\sqrt{K}+ x }-\frac{1}{3}\Bigr)
\\
&&
\qquad
=
P\Bigl(\frac{R^{[K]}-K/3}{\sqrt{K}}\geq -\frac{x\sqrt{K}}{3(3\sqrt{K}+ x) }\Bigr).
\feq
Taking the limit as $K$ tends to infinity and inserting $\sigma(R^{[K]})=\frac{1}{3}\sqrt{\frac{2K}{5}},$ yields
\beq
\lim_{K\to\infty}P\Bigl(\sqrt{K}(<D_{\centerdot ,K}>_p-2)\leq x\Bigr)
=\lim_{K\to\infty}P\Bigl(\frac{K/3-R^{[K]}}{\sigma(R^{[K]})}\leq x\sqrt{\frac{5}{18}}\Bigr),
\feq
as required.
\par
In particular, Theorem~\ref{dltn} implies the weak law of large numbers for $<D_{\centerdot ,K}>_p:$
\beq
<D_{\centerdot ,K}>_p\,\overset{P}{\to}\, 2, \qquad \mbox{\rm as}~K\to\infty,
\feq
where $\overset{P}{\to}$ indicates the convergence in probability. Interestingly enough,
this law of large numbers is consistent with the ``mean-field" \eqref{mf}. Heuristically, this may be explained by the CLT
for $D_{i,K}$ stated as a conjecture in Section~\ref{clti}, which implies that for large values of $K,$ with high probability the value of $D_{i,k}$ is close to its expectation. This is also consistent with the heuristic Ginzburg criterion \cite{ginzburg} asserting that a mean-field approximation may work suitably in the situation when the variance of the underlying parameter is of a smaller order than its average square. In our case, in view of \eqref{mf},
\beq
E\Bigl[\Bigl(\sum_{i=1}^{K-1}iD_{i,K}\Bigr)^2\Bigr]\geq \Bigl[ E\Bigl(\sum_{i=1}^{K-1}iD_{i,K}\Bigr)\Bigr]^2=\frac{4K^2}{9},
\feq
while by Theorem~\ref{cltn}, $\sigma^2\Bigl(\sum_{i=1}^{K-1}iD_{i,K}\Bigr)=\sigma^2\bigl(R^{[K]}\bigr)=\frac{2K}{45}.$
\begin{figure*} [ht!]
		\centering
		\begin{tabular}{cccc}
			\includegraphics[width=0.22\textwidth]{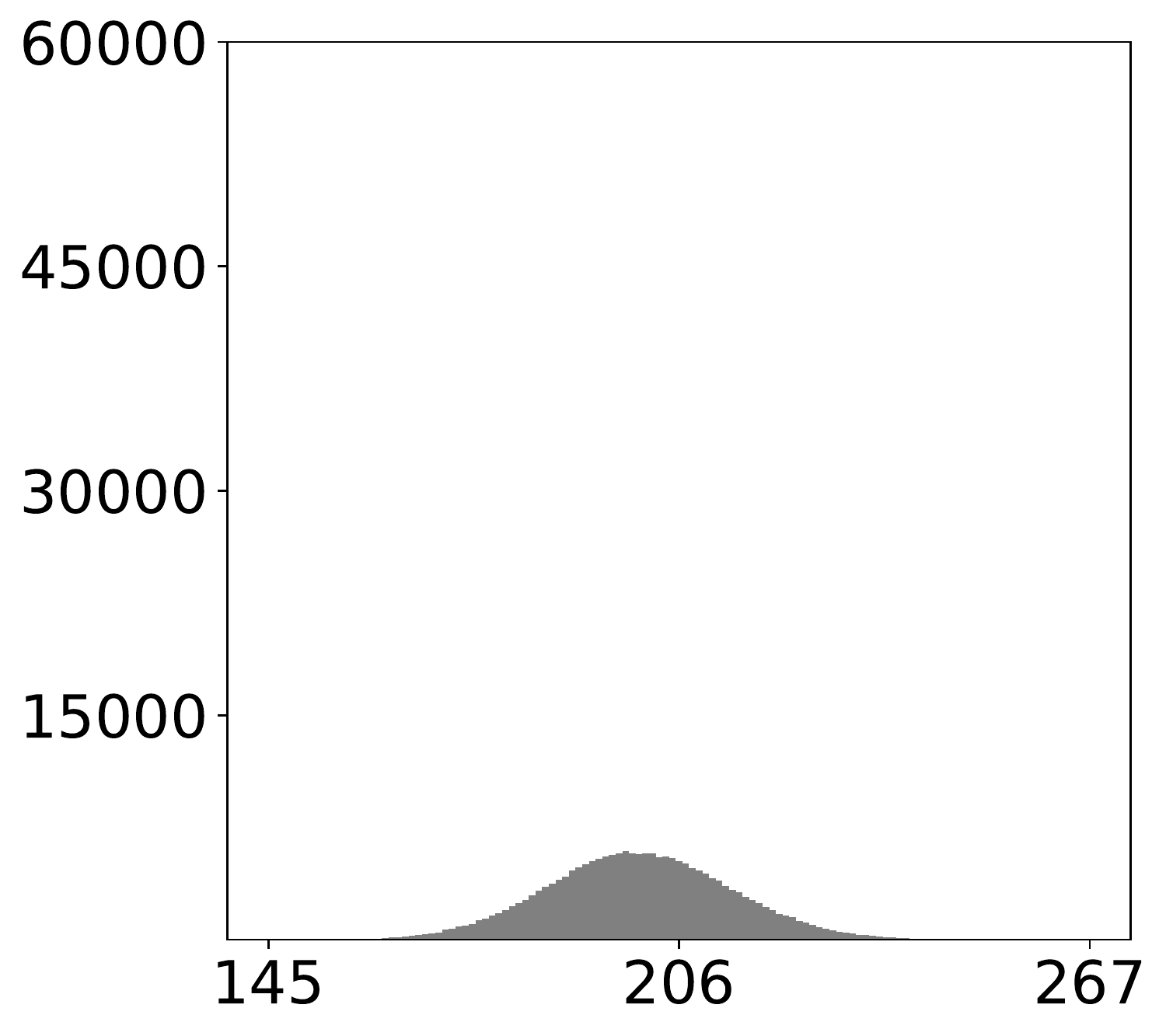} &
			\includegraphics[width=.22\textwidth]{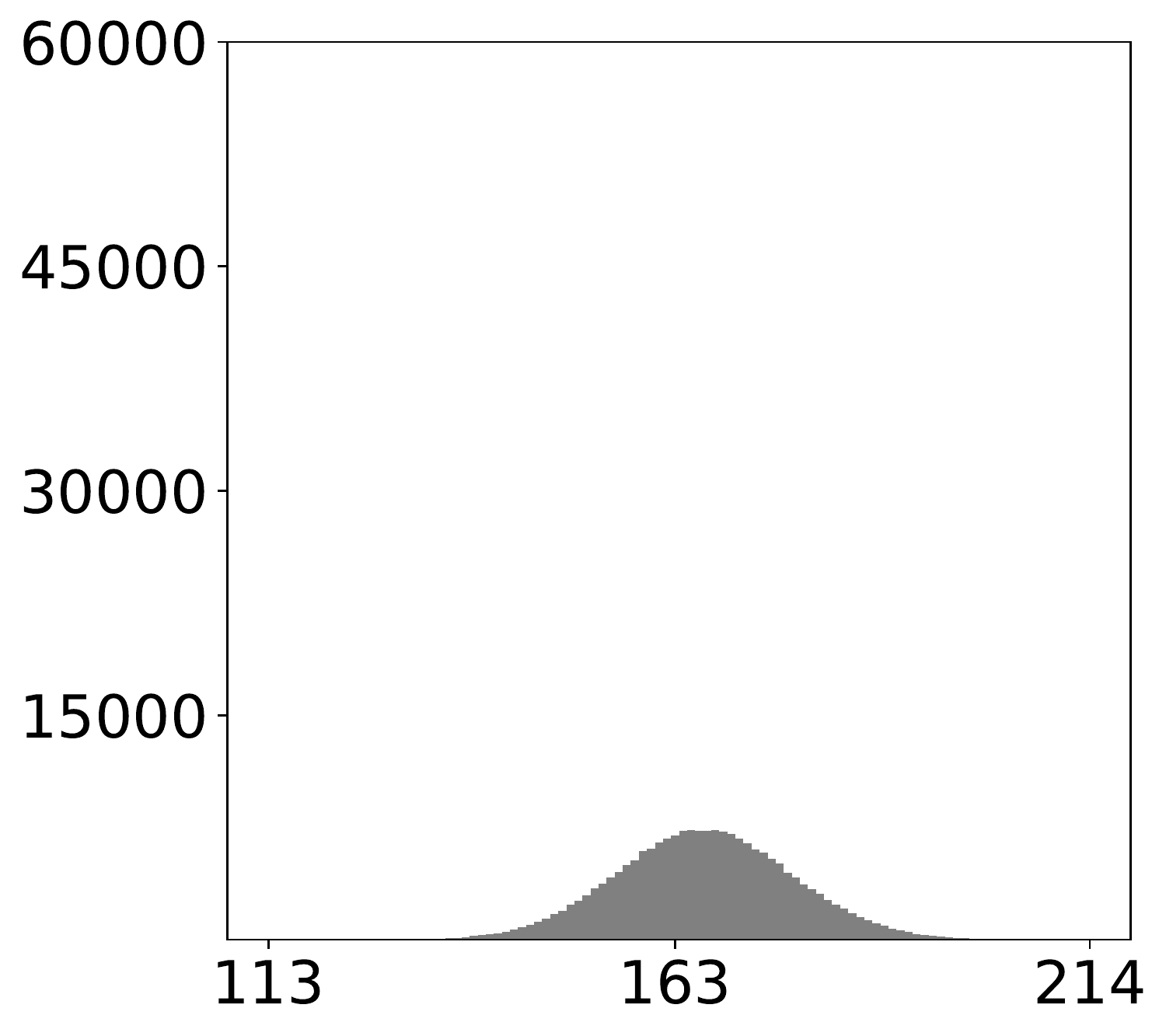} &
			\includegraphics[width=0.22\textwidth]{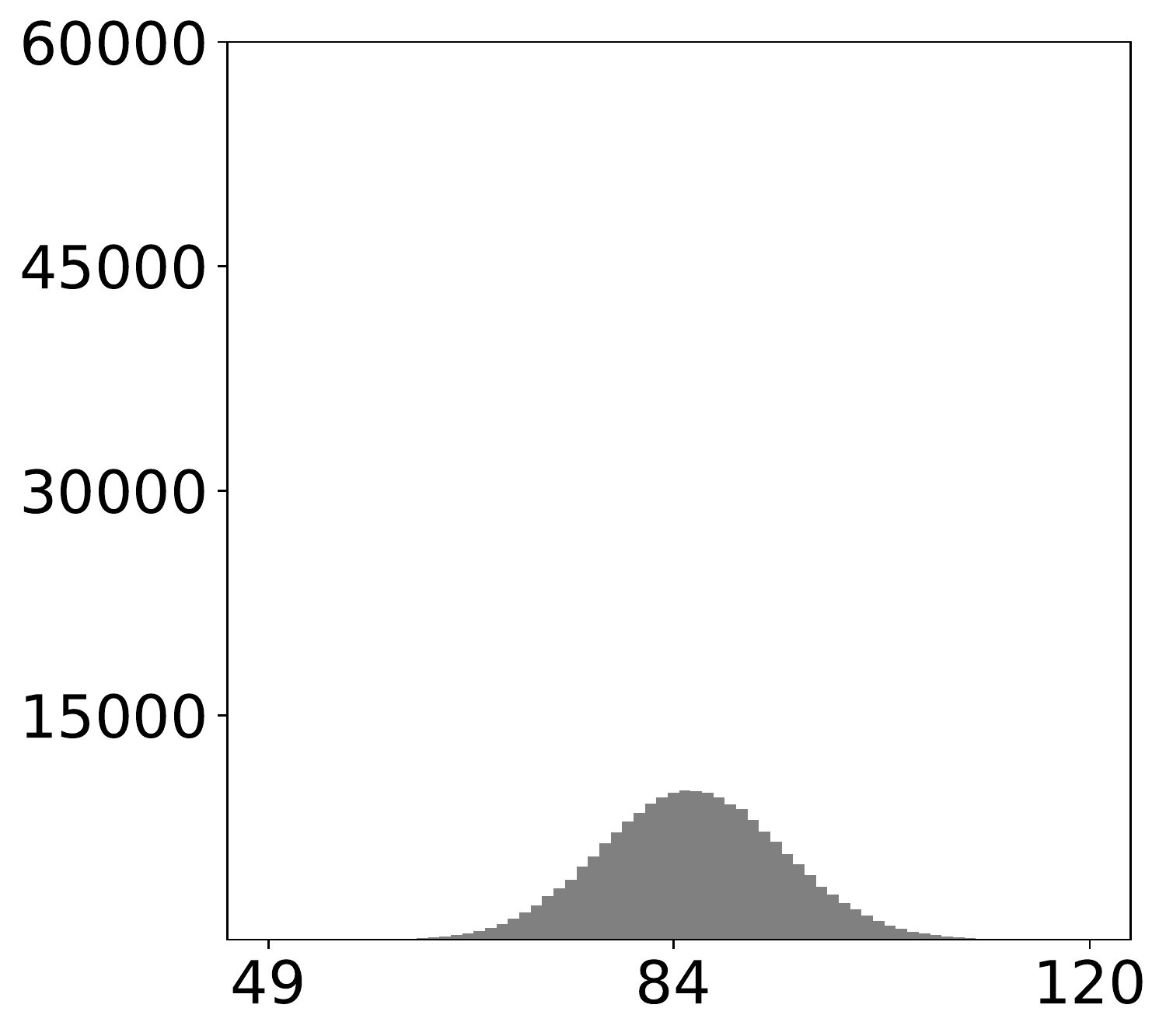} \\
            (a) $D_{1,1500}$  & (b) $D_{2,1500}$ & (c) $D_{3,1500}$ \\
			\includegraphics[width=.22\textwidth]{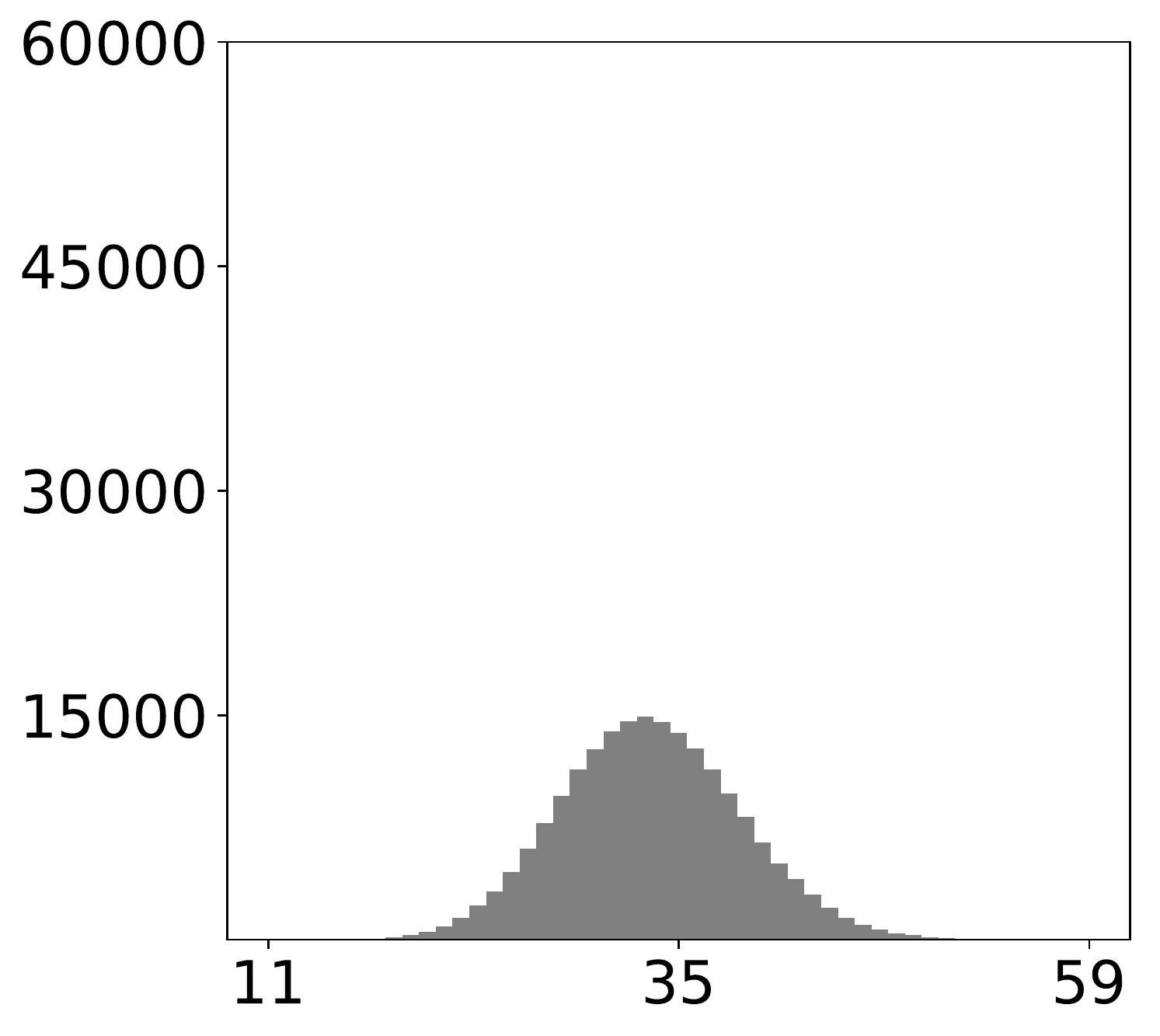} &
			\includegraphics[width=.22\textwidth]{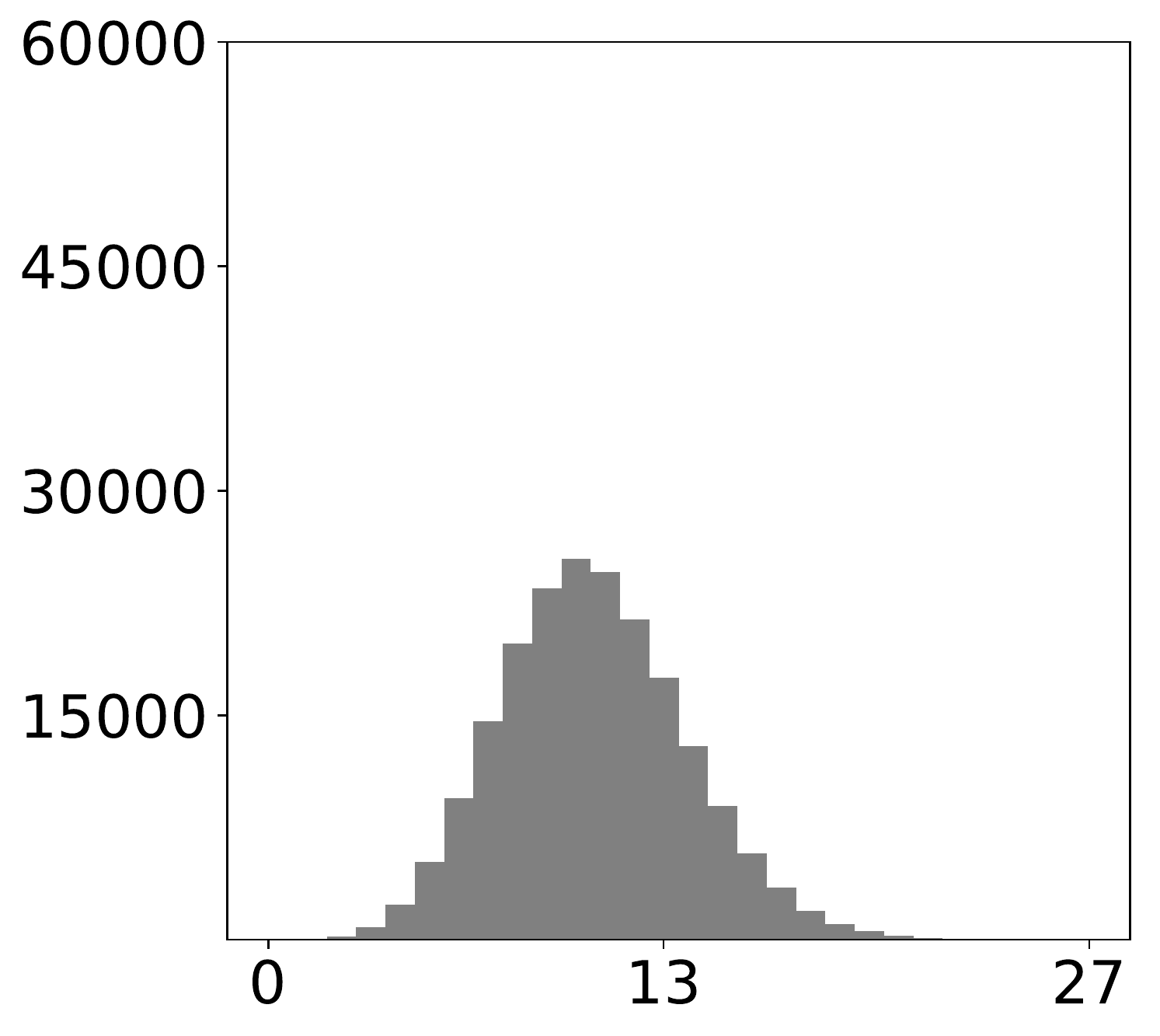} &
			\includegraphics[width=0.22\textwidth]{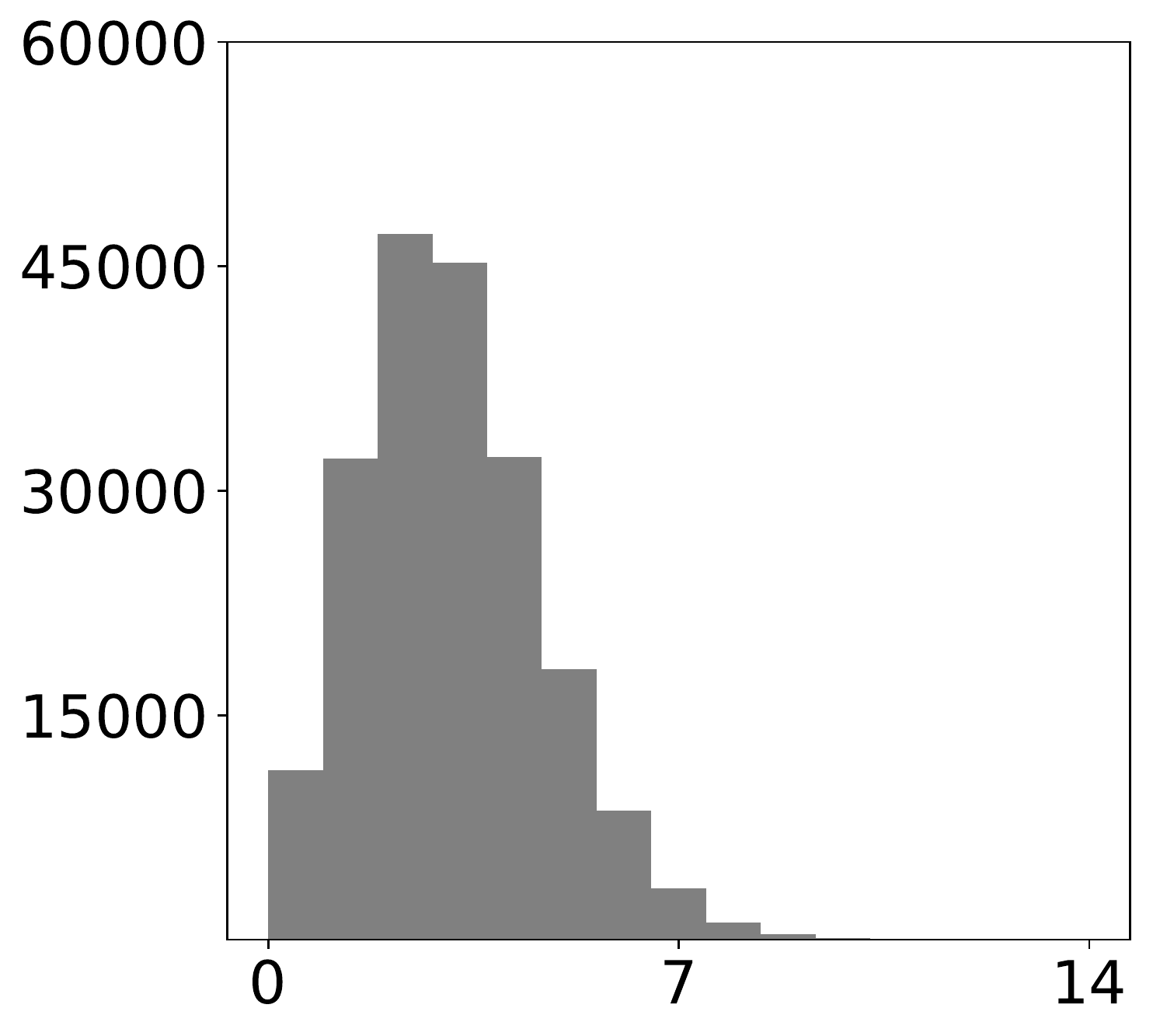}  \\
			 (d) $D_{4,1500}$ & (e) $D_{5,1500}$  & (f) $D_{6,1500}$ \\
		\end{tabular}
		\caption{Empirical distribution of $D_{i,K}$ for several values of $i$ and $K=1500,$ each obtained via a simulation of $200,000$ runs.}
		\label{fig:simul200k}
	\end{figure*}	
	\subsection{Moments of $D_{i,K}:$ general recursions}
\label{moments}
Our numerical simulations strongly suggest that for all $i\in \nn,$ a properly scaled $D_{i,K}$ converges to a normal law as $K$ tends to infinity.
A formal conjecture in this regard is stated in Section~\ref{clti} below. See Figure~\ref{fig:simul200k} for a histogram of the empirical distribution of $D_{i,K}$ for several values of $i$ and $K$ obtained in a simulation of $200,000$ simulations of the model. In this section we devise a method for estimation of the first two moments of $D_{i,K}.$ We believe that both the expectation and the variance of this random variable grow linearly with $K$ (see Section~\ref{clti}), in what follows we will verify this conjecture analytically for all $i\leq 7.$
\par
Fix $i\in\nn $ and assume that $K>i.$ Recall the heights $H_n(k)$ from \eqref{hok} and \eqref{hk}. Let $s,l,r,k\in\nn$ be such that
\beq
2\leq s\leq s+l< s+k-r<k+s<K.
\feq
Some of the above inequalities are trivially hold for natural numbers, they are illustrated altogether in Fig.~\ref{fig:Dk}.
Assume that at some time $T\in\nn$ we have (see Fig.~\ref{fig:Dk})
\beq
\begin{array}{ll}
H_T(j)=1&~\mbox{\rm if}~j\in\{s-1,s+k\},
\\
H_T(j)>1&~\mbox{\rm if}~j\in\{s,s+1,\ldots,s+l-1\}\cap \{s+k-r,s+1,\ldots,s+k-1\},
\\
H_T(j)=0&~\mbox{\rm if}~j\in\{s+l,\ldots,s+k-r-1\}.
\end{array}
\feq
Note that we do not specify the values of $H_T(j)$ for $j<s-1$ and $j>s+k.$ In words, at time $T$ we have a root at site $s-1$
followed by a block of particles not touching the ground of length $l,$ followed by an interval empty of particles, which is followed
by a block of particles not touching the ground of length $r,$ and ends with a root at point $s+k.$
\par
We define $\cald_{l,s,j, i}$ as the number of pairs of consecutive roots in the interval $[s-1,s+k]$ with the distance $i$ between them.
It is not hard to verify that the distribution of $\cald_{l,s,j, i}$ is independent of $T$ and the configuration of particles at time $T.$ In particular,
	\beqn
	\call(D_{i, K}) = \call\left(\cald_{0,0,K-1, i}\right). \label{DCalD}
	\feqn
In what follows we derive and study a system of equations for $\cald_{l,r,K, i},$ and then extract an appropriate information for $D_{i,K}$ from these equations. The above construction is considerably more involved comparing to the auxiliary process exploited in Section~\ref{sec:roots}. The reason why we are using this construction is that the distribution of $\cald_{l,r,K, i}$ does vary with $l$ and $r$ because of the effect of the corner, and hence
$l$ and $r$ should be taken into consideration in some way. By the corner effect we mean that the distribution of the
distance between the corner root $s-1$ and next to it root within the interval depends on the values of the parameter $l.$ Similarly, the distribution of the distance between the corner root $s+k$ and next to it root within the interval depends on $r.$	
	\begin{figure*}[t!]
		\centering
		\includegraphics[width=.6\textwidth]{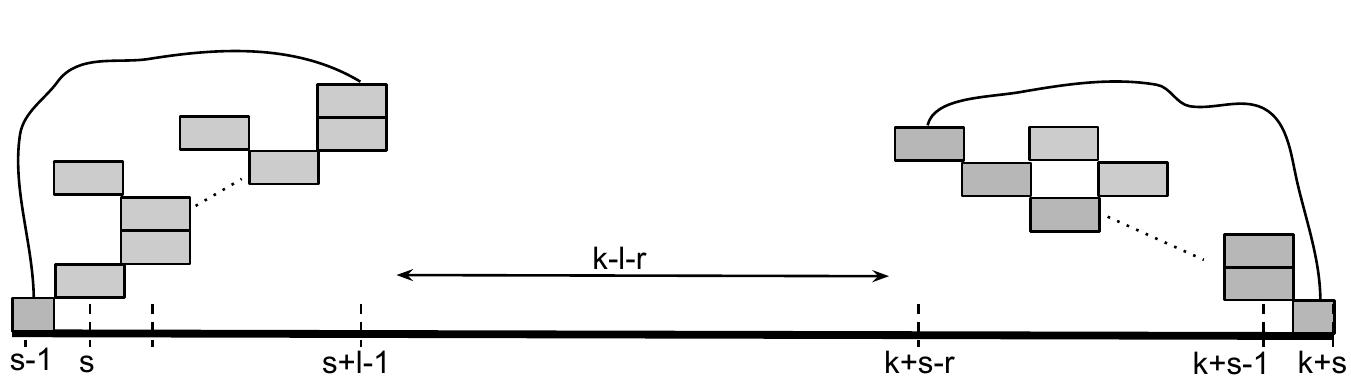}
		\caption{Example of configuration at time $T$ in the auxiliary construction defining $\cald_{l,r,k, i}.$}
		\label{fig:Dk}
	\end{figure*}
\par
For all $k\geq l+ r + 3$,
	\beqn \label{dk_recursion}
	&&	\call(\cald_{l,r,k, i}) \\
	\nonumber
	&&
	\quad
	 =\call \left( \odin_{\{Y = 1 \}}\cald^{(1,1)}_{l+1,r,k, i} + \sum_{j=2}^{k-r-l-1} \odin_{\{Y= j \}} (\cald^{(1,j)}_{l,0,j+l-1, i} + \cald^{(2,j)}_{0,r,k-j-l, i}) + \odin_{\{Y = k-r-l \}}\cald^{(1,k)}_{l,r+1,k, i}\right)
	\feqn
	where
\begin{itemize}
			\item [-] $Y$ is distributed uniformly over the interval of integers  $[1,k-r-l],$
			\item [-] $\call(\cald_{l, r, k, i}) = \call(\cald_{l, r, k, i}^{(1,j)}) = \call(\cald_{l, r, k, i}^{(2,j)})$
for all $k\in [K],$ and $l,r \in \nn,$
			\item [-] $\cald_{l, r, k, i},$ $\cald_{l, r, k, i}^{(1,j)},$ $\cald_{l, r, k, i}^{(2,j)}$ and $Y_K$ are independent of each other for all values of the arguments $l,r,k,i$ and $j.$
		\end{itemize}
	In addition, we have
	\beqn \label{dk_recursion1}
	\call(\cald_{l,r,k, i})
	= \odin_{\{k=i\}}, &\quad \text{if} \quad l + r \in \{k, k-1, k-2\},
	\feqn
	for the initial condition of the system.
	See Figure \ref{fig:d_partition} for a visual explanation of \eqref{dk_recursion}.
	\begin{figure*}[t!]
		\centering
		\begin{tabular}{cc}
			\includegraphics[width=.60\textwidth]{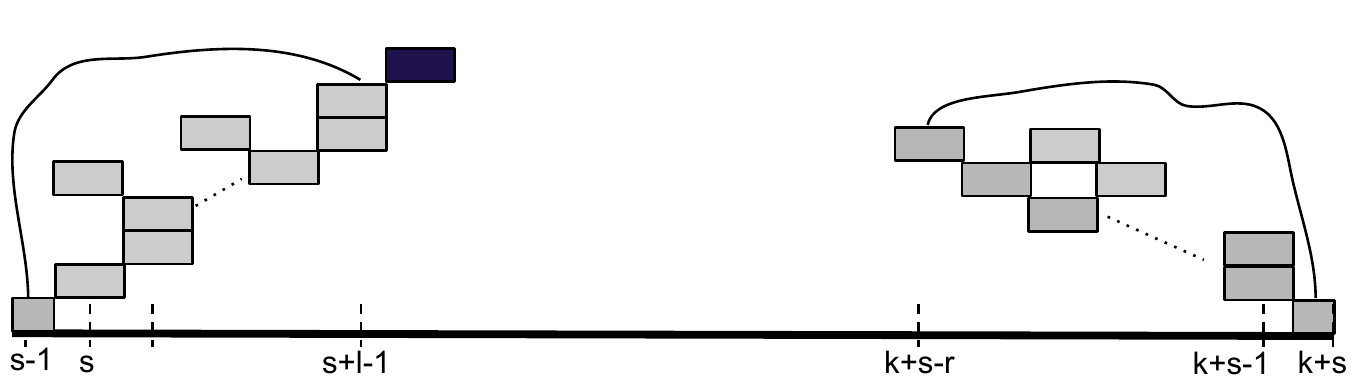} \\ \includegraphics[width=.60\textwidth]{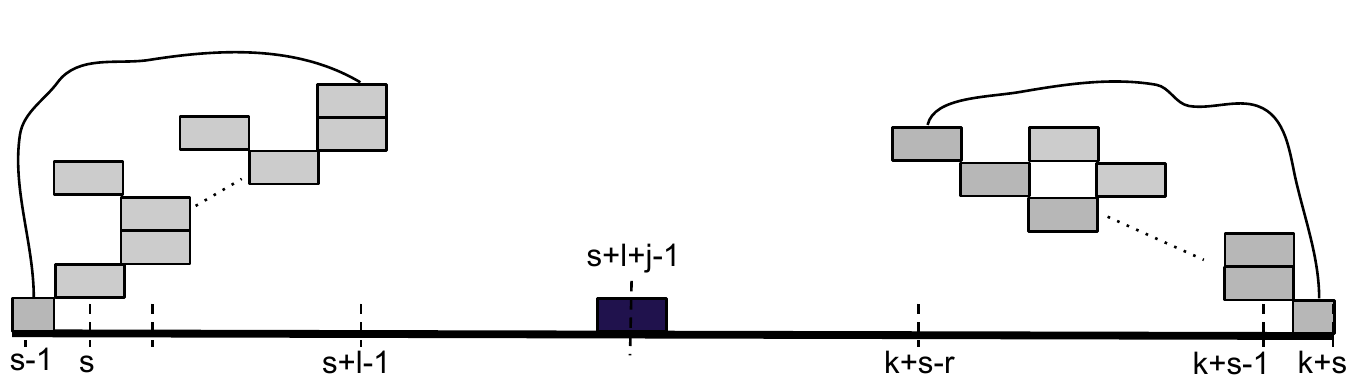} \\
			\includegraphics[width=.60\textwidth]{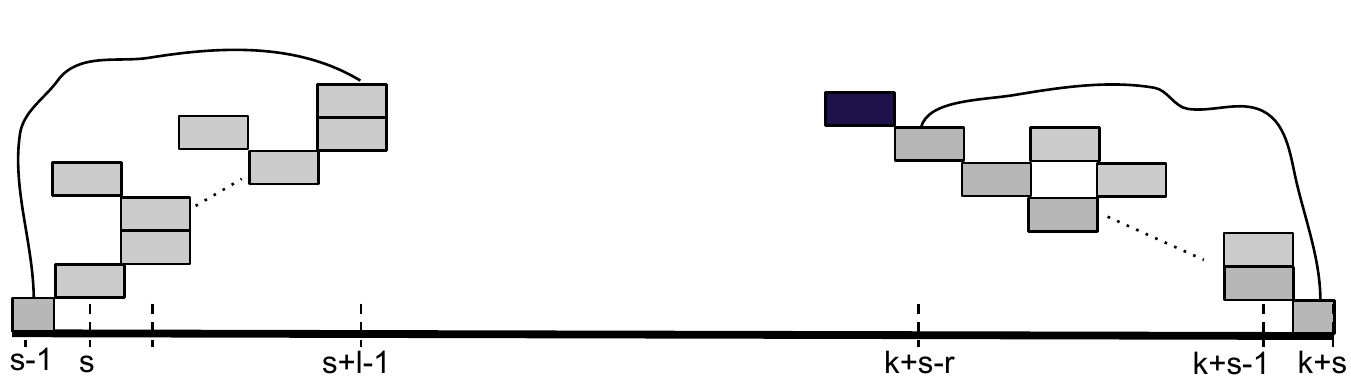}
		\end{tabular}
		\caption{The picture on the top corresponds to $Y=1,$ in the middle corresponds to $Y\in \{ 2,\cdots,k-l-r-1 \},$ and in the bottom to $Y=k-r-l$.}
		\label{fig:d_partition}
	\end{figure*}
\par
	One can rewrite \eqref{dk_recursion} as
	\beqn \label{dk_recursion_gen}
	&&E(u^{\cald_{l,r,k,i}})
	\\
	\nonumber
	&&
	\quad
	=\frac{1}{k-l-r} \left( E(u^{\cald_{l+1,r,k,i}}) + \sum_{j=2}^{k-r-l-1}  E(u^{\cald_{l,0,j+l-1,i}})E(u^{\cald_{0,r,k-j-l,i}}) + E(u^{\cald_{l,r+1,k,i}}) \right)
	\feqn
	for all $l+r\leq k-3.$ Similarly, \eqref{dk_recursion1} can be written as
	\beqn \label{dk_recursion1_gen}
	E(u^{\cald_{l,r,k, i}})	= u^{\odin_{\{k=i\}}}, \quad l + r \in \{k, k-1, k-2\}.
	\feqn
	To solve the system of equations \eqref{dk_recursion_gen} and \eqref{dk_recursion1_gen} for all possible $k, l, r$, we  define the following generating function with the domain in $\cc^4:$
	\beqn \label{Dxyzu}
	D_{i}(x,y,z,u) = \sum_{K=3}^\infty \sum_{l=0}^{K-3} \sum_{r=0}^{K-3-l} E(u^{\cald_{l,r,K, i}})y^lz^rx^{K-l-r}.
	\feqn
	Notice that
	\beqn
	&&D_{i}(0,y,z,u)= 0, \nonumber \\
	&&D_{i}(x,y,0,u) =  \sum_{K=3}^{\infty} \sum_{l=0}^{K-3} E(u^{\cald_{l,0,K, i}})x^{K-l}y^l,  \nonumber \\
	&&D_{i}(x,0,z,u) =  \sum_{K=3}^{\infty} \sum_{r=0}^{K-3}  E(u^{\cald_{0,r,K, i}})x^{K-r}z^r, \nonumber \\
	&&D_{i}(x,0,0,u) =  \sum_{K=3}^{\infty} E(u^{\cald_{0,0,K, i}})x^K = \sum_{K=3}^{\infty} E(u^{D_{i, K+1}})x^{K}. \label{D00}	
	\feqn
	Here we used the usual convention $0^0=1.$  Inserting \eqref{dk_recursion_gen} and \eqref{dk_recursion1_gen} into \eqref{Dxyzu} we obtain
	\beqn
	\label{D3_first}
	&& D_{i}(x,y,z,u) = \sum_{K=3}^\infty \, \sum_{l+r \leq K-3} \frac{1}{K-l-r} \times  \\ \nonumber
	&& \quad \left( E(u^{\cald_{l+1,r,K,i}}) + \sum_{j=2}^{K-r-l-1}  E(u^{\cald_{l,0,j+l-1,i}})E(u^{\cald_{0,r,K-j-l,i}}) + E(u^{\cald_{l,r+1,K,i}}) \right)x^{K-l-r}y^lz^r.
	\feqn
	Let $A_{i}(x,y,z,u) = \sum_{K=3}^\infty \, \sum_{l+r \leq K-3} \frac{1}{K-l-r} E(u^{\cald_{l+1,r,K,i}})x^{K-l-r}y^lz^r$
	to be the first summation term in \eqref{D3_first}. Then
	\beqn
	\label{Ax}
	&&	y \frac{\partial A_i(x,y,z,u)}{\partial x} = \sum_{K=3}^\infty \sum_{l=0}^{K-3} \sum_{r=0}^{K-3-l} E(u^{\cald_{l+1,r,K,i}})x^{K-(l+1)-r}y^{l+1}z^r
	\\
	\nonumber 	&& \qquad= \sum_{K=3}^\infty \sum_{l=1}^{K-2} \sum_{r=0}^{K-2-l} E(u^{\cald_{l,r,K,i}})x^{K-l-r}y^{l}z^r
	\\
	&& \qquad = \sum_{K=3}^\infty \sum_{l=0}^{K-3} \sum_{r=0}^{K-3-l} E(u^{\cald_{l,r,K,i}})x^{K-l-r}y^{l}z^r  + x^2 \sum_{K=3}^\infty \sum_{l=0}^{K-3} u^{\odin_{\{K=i\}}}y^{l}z^{K-2-l} \nonumber \\
	&& \qquad  \quad   - \sum_{K=3}^\infty  \sum_{r=0}^{K-3} E(u^{\cald_{0,r,K,i}})x^{K-r}z^r - \sum_{K=3}^\infty  u^{\odin_{\{K=i\}}}x^{2}z^{K-2} + \sum_{K=3}^\infty  u^{\odin_{\{K=i\}}} x^{2}y^{K-2}\nonumber \\	\nonumber
	&&\qquad = D_i(x,y,z,u) - D_i(x,0,z,u) + x^2\sum_{K=3}^\infty  u^{\odin_{\{K=i\}}} \Bigl( \sum_{l=0}^{K-3}y^{l}z^{K-2-l} - z^{K-2} + y^{K-2} \Bigr).
	\feqn
	Let $B_i(x,y,z,u) = \sum_{K=3}^\infty\, \sum_{l+r \leq K-3} \frac{1}{K-l-r} E(u^{\cald_{l,r+1,K,i}})x^{K-l-r}y^lz^r$
be the third summation term in $\eqref{D3_first}.$ Then
	\beqn
	z \frac{\partial B_i(x,y,z,u)}{\partial x} &=& D_i(x,y,z,u) - D_i(x,y,0,u) \nonumber \\
	&&+ x^2\sum_{K=3}^\infty  u^{\odin_{\{K=i\}}} \Bigl( \sum_{r=0}^{K-3}y^{K-2-r}z^{r} - y^{K-2} + z^{K-2} \Bigr). \label{Bx}
	\feqn
	Finally, we investigate the second summand in $\eqref{D3_first}.$ Let
	\beq
	C_{i}(x,y,z,u) =  \sum_{K=3}^\infty \sum_{l+r \leq K-3} \frac{1}{K-l-r} \sum_{j=2}^{K-r-l-1} E(u^{\cald_{l,0,j+l-1,i}})E(u^{\cald_{0,r,K-j-l,i}}) x^{K-l-r} y^l z^r.
	\feq
	Then
	\beq
	\frac{\partial C_i(x,y,z,u)}{\partial x} &=& \sum_{K=3}^\infty \sum_{l=0}^{K-3}\sum_{r=0}^{K-3-l}  \sum_{j=2}^{K-r-l-1} E(u^{\cald_{l,0,j+l-1,i}})E(u^{\cald_{0,r,K-j-l,i}}) x^{K-r-l-1} y^l z^r \\
	&=& \sum_{K=3}^\infty \sum_{l=0}^{K-3} \sum_{r=0}^{K-3-l} \sum_{j=2+l}^{K-r-1} \left( E(u^{\cald_{l,0,j-1,i}})x^{j-1-l}y^l \right) \left(E(u^{\cald_{0,r,K-j,i}}) x^{K-j-r} z^r \right).
	\feq
	Changing the order of the summations, we obtain
	\beqn
	&&\frac{\partial C_i(x,y,z,u)}{\partial x} = \left(\sum_{j=2}^\infty \sum_{l=0}^{j-2} E(u^{\cald_{l,0,j-1,i}})x^{j-1-l}y^l \right) \left( \sum_{K=j+1}^{\infty} \sum_{r=0}^{K-j-1}  E(u^{\cald_{0,r,K-j,i}}) x^{K-j-r} z^r \right)  \nonumber \\
	&& \qquad = \left(\sum_{j=1}^\infty \sum_{l=0}^{j-1} E(u^{\cald_{l,0,j,i}})x^{j-l}y^l \right) \left( \sum_{K=1}^{\infty} \sum_{r=0}^{K-1}  E(u^{\cald_{0,r,K,i}}) x^{K-r} z^r \right) \nonumber \\
	&& \qquad = \left(D_{i}(x,y,0,u) + u^{\odin_{\{i=1\}}}x + \sum_{j\geq 2} u^{\odin_{\{i=j\}}}xy^{j-2}(y+x)   \right) \nonumber \\
	&& \qquad \times \left(D_{i}(x,0,z,u) + u^{\odin_{\{i=1\}}}x + \sum_{j\geq 2} u^{\odin_{\{i=j\}}}xz^{j-2}(z+x)  \right). \label{Cx}
	\feqn
	Taking the derivative with respect to $x$ on both sides of $\eqref{D3_first}$ and combining the result with \eqref{Ax}, \eqref{Bx}, and \eqref{Cx}, yields the following system of equations:
	\beqn
	\nonumber
	&& yz\frac{\partial D_{i}(x,y,z,u)}{\partial x} =
	\\
	&& \quad z \left( D_i(x,y,z,u) - D_i(x,0,z,u)
	+ x^2\sum_{K=3}^\infty  u^{\odin_{\{K=i\}}} \Bigl( \sum_{l=0}^{K-3}y^{l}z^{K-2-l} - z^{K-2} + y^{K-2} \Bigr)\right) \nonumber\\
	&& \quad~ + y \left(D_i(x,y,z,u) - D_i(x,y,0,u)
	+ x^2\sum_{K=3}^\infty  u^{\odin_{\{K=i\}}} \Bigl( \sum_{r=0}^{K-3}y^{K-2-r}z^{r} - y^{K-2} + z^{K-2} \Bigr) \right) \nonumber \\
	&& \quad~  + yz \left(D_{i}(x,y,0,u) + u^{\odin_{\{i=1\}}}x + \sum_{j\geq 2} u^{\odin_{\{i=j\}}}xy^{j-2}(y+x)   \right) \nonumber\\
	&& \quad ~ \qquad  \times \left(D_{i}(x,0,z,u) + u^{\odin_{\{i=1\}}}x + \sum_{j\geq 2} u^{\odin_{\{i=j\}}}xz^{j-2}(z+x)  \right). \label{DiPDE}
	\feqn
	
	\subsection{Moments of $D_{i,K}:$ case $i=1$}
	\label{i1}
	When  $i=1,$ \eqref{DiPDE} reduces to
	\begin{align*}
	D_1(x,y,z,u)&=&\frac{(u-1)^2yz-u(u-1)(y+z)+u^2+2}{3(1-y)(1-z)}x^3
	+\frac{(1-u)(y+z)+2u+1}{3(1-y)(1-z)}x^4\\
	&&+\frac{2u(u-1)^2yz-(u-1)(2u^2+3)(y+z)+7+6u+2u^3}{15(1-y)(1-z)}x^5+\cdots.
	\end{align*}
	An inspection of the terms in right-hand side reveals that $D_1(x,y,z,u)$ may have the following form:
	\beqn \label{D1functionalForm}
	D_1(x,y,z,u)=\sum_{K\geq3}\frac{a_K(u)yz+b_K(u)(y+z)+c_K(u)}{(1-y)(1-z)}x^K,
	\feqn
	where $a_K(u),b_K(u),c_K(u)$ are polynomials in $u.$ In what follows we confirm this guess.
	\par
	First observe that
\beqn \label{D1xyz1} D_1(x,y,z,1)=\frac{x^3}{(1-x)(1-y)(1-z)}=\sum_{K\geq3}\frac{x^K}{(1-y)(1-z)},
\feqn
which implies that $a_K(1)=0$, $b_K(1)=0$ and $c_K(1)=1$, for all $K\geq3$. Similarly, letting $y=z=0,$ we obtain from the identity
	\beqn \label{D00abc}
	c(x,u) := D_1(x,0,0,u)=\sum_{K\geq3}c_K(u)x^K,
	\feqn
	that
\beqn
\label{ckey}
c_K(u) = E(u^{\cald_{0,0,K,1}}) = E(u^{D_{1, K+1}}).
\feqn
	Next, we substitute the functional form \eqref{D1functionalForm} of $D_1(x,y,z,u)$ into \eqref{DiPDE}. After a few simple algebraic manipulations, grouping, and comparing coefficients on both sides \eqref{DiPDE}, we arrive to the following system of recurrence equations:
	\begin{align}
	(K+1)a_{K+1}(u)&=(1-u)^2\odin_{\{K=2\}}+\sum_{j=3}^{K-3}b_j(u)b_{K-j}(u)+2(1-u)b_{K-1}(u),\label{eqrabc1}\\
	(K+1)b_{K+1}(u)&=a_K(u)+b_K(u)+\sum_{j=3}^{K-3}b_j(u)c_{K-j}(u)+ub_{K-1}(u)+b_{K-2}(u)\nonumber\\
	&+(1-u)c_{K-1}(u)+u(1-u)\odin_{K=2}+(1-u)\odin_{\{K=3\}},\label{eqrabc2}\\
	(K+1)c_{K+1}(u)&=2b_K(u)+2c_K(u)+\sum_{j=3}^{K-3}c_j(u)c_{K-j}(u)+2uc_{K-1}(u)+2c_{K-2}(u)\nonumber\\
	&+(2+u^2)\odin_{\{K=2\}}+2u\odin_{\{K=3\}}+\odin_{\{K=4\}},\label{eqrabc3}
	\end{align}
	for all $K\geq2$.
	\par
	Using induction, one can now verify that $a_K(u), b_K(u), c_K(u)$ are all polynomials of degree $\frac{2K-1-3(-1)^K}{4}$ for $K\geq3$. Several first values of these polynomials are given in Table \ref{tababc}.
	
	\begin{remark}
		The above recurrence equations can be equivalently written as the following system of differential equations. Define $$A(x,u)=\sum_{K\geq3}a_K(u)x^K \quad \text{and} \quad B(x,u)=\sum_{K\geq3}b_K(u)x^K.$$ Then by multiplying each equation in \eqref{eqrabc1}-\eqref{eqrabc3} by $x^K$ and summing over $K\geq2$, we get
		\begin{align*}
		\frac{\partial}{\partial x}A(x,u)&=(1-u)^2x^2+B(x,u)(B(x,u)+2(1-u)x),\\ \frac{\partial}{\partial x}B(x,u)&=A(x,u)+B(x,u)+(C(x,u)+x(u+x))(B(x,u)+(1-u)x),\\ \frac{\partial}{\partial x}C(x,u)&=2B(x,u)+2C(x,u)+C^2(x,u)+2x(u+x)C(x,u)+(2+u^2)x^2+2ux^3+x^4.
		\end{align*}
	\end{remark}
	
	\begin{table}
		\begin{adjustbox}{width=\columnwidth,center}
			\begin{tabular}{r|lll}
				\multicolumn{1}{l}{$K$}
				& \multicolumn{1}{l}{$a_K(u)$}
				& \multicolumn{1}{l}{$b_K(u)$}
				& \multicolumn{1}{l}{$c_K(u)$} \\ \cline{2-4}
				$3$ & $\frac{1}{3}(1-2u+u^2)$& $\frac{1}{3}(u-u^2)$&$\frac{1}{3}(2+u^2)$\\
				$4$ & $0$& $\frac{1}{3}(1-u)$&$\frac{1}{3}(1+2u)$  \\
				$5$ &$\frac{1}{15}(2u-4u^2+2u^3)$& $\frac{1}{15}(3-3u+2u^2-2u^3)$&$\frac{1}{15}(7+6u+2u^3)$\\
				$6$ &$\frac{1}{9}(1-2u+u^2)$& $\frac{1}{45}(4+7u-11u^2)$&$\frac{1}{45}(20+8u+17u^2)$\\
				$7$ &$\frac{1}{315}(18-36u+35u^2-34u^3+17u^4)$& $\frac{1}{315}(45-2u-43u^2+17u^3-17u^4)$&$\frac{1}{315}(98+132u+68u^2+17u^4)$	
			\end{tabular}%
		\end{adjustbox}
		\caption{The values of $a_K(u),b_K(u),c_K(u)$, for $K=3,4,5,6,7.$} \label{tababc}
	\end{table}
	Even though we were unable to solve \eqref{eqrabc1}, \eqref{eqrabc2}, and \eqref{eqrabc3} directly, we can leverage them to compute the
first two moments of $D_{1,K}$ for an arbitrary $K\geq 3.$
	
	\begin{theorem} \label{thm:D1K}
		The following holds true for $D_{1,K}:$
		\begin{itemize}
			\item [(i)] $E(D_{1,4}) =\frac{2}{3}$ and
			\beqn
			\label{e}
			E(D_{1,K}) =\frac{2K}{15}\qquad \forall\,K\geq 5.
			\feqn
			\item [(ii)] $\sigma^2(D_{1,4}) =\frac{8}{9},$ $\sigma^2(D_{1,5}) =\frac{2}{9},$ $\sigma^2(D_{1,6}) =\frac{24}{25},$
			$\sigma^2(D_{1,7}) =\frac{184}{225},$  $\sigma^2(D_{1,8}) =\frac{1588}{1575},$ and
			\beqn
			\label{d} \sigma^2(D_{1,K}) =
			\frac{1772K}{14175} \qquad \forall\,K\geq 9.
			\feqn		
		\end{itemize}
	\end{theorem}
Since $\sigma^2(D_{1,K})\sim E(D_{1,K})$ as $K\to\infty,$ Chebyshev's inequality yields
\begin{corollary}
\label{cln}
$\frac{15D_{1,K}}{2K}\overset{P}{\to}1$ as $K$ tends to infinity.
\end{corollary}
	\begin{proof}[Proof of Theorem~\ref{thm:D1K}]
		\item [(i)]
		Define $a'_K=\frac{d}{du}a_K(u)\mid_{u=1}$, $b'_K=\frac{d}{du}b_K(u)\mid_{u=1}$ and $c'_K=\frac{d}{du}c_K(u)\mid_{u=1}$.
		Note that $a_K(1)=b_K(1)=0$ and $c_K(1)=1$ for all $K\geq3$.
		By differentiating \eqref{eqrabc1} at $u=1$, we obtain
		$a'_K=0$. Thus, in view of \eqref{eqrabc2} and \eqref{eqrabc3}, we have for all $K\geq4$,
		\begin{align*}
		(K+1)b'_{K+1}=-1+\sum_{j=3}^Kb'_j, \quad \text{and} \quad (K+1)c'_{K+1}=2b'_K+2+2\sum_{j=3}^{K}c'_j,
		\end{align*}
		with $b'_3=b'_4=-1/3$ and $c'_3=c'_4=2/3$. By induction, $b'_K=-1/3$ and $c'_K=\frac{2K+2}{15}$ for all $K\geq4$. Therefore, for $K\geq4$,
		\begin{align*}
		\frac{d}{du}a_K(u)\mid_{u=1}=0,\qquad \frac{d}{du}b_K(u)\mid_{u=1}=-\frac{1}{3}, \quad \text{and} \quad \frac{d}{du}c_K(u)\mid_{u=1}&=\frac{2(K+1)}{15}.
		\end{align*}
		This along with \eqref{D00abc} and \eqref{ckey} implies that
		\begin{align*}
		\sum_{K\geq 3} E(D_{1,K+1})x^K = \frac{\partial}{\partial u}D_1(x,0,0,u)\mid_{u=1}
		&=\frac{2}{3}x^3+\sum_{K\geq4}\frac{(2K+2)}{15}x^K.
		\end{align*}
		Then \eqref{DCalD} gives the result for the expected values.
		\item [(ii)]
		Similarly, in order to calculate the variance, we consider
		\beq
		a''_K=\frac{d^2}{du^2}a_K(u)\mid_{u=1}, \quad b''_K=\frac{d^2}{du^2}b_K(u)\mid_{u=1}, \quad \mbox{\rm and} \quad c''_K=\frac{d^2}{du^2}c_K(u)\mid_{u=1}.
		\feq
		Note that $a_K(1)=b_K(1)=0$ and $c_K(1)=1$ for all $K\geq3$.
		By differentiating the equations in \eqref{eqrabc1}-\eqref{eqrabc3} twice, letting $u=1$, and using induction on $K,$ we obtain that
		\beq
		&&a''_3=\frac{2}{3},\, a''_4=0,\, a''_5=\frac{4}{15},\, a''_k =\frac{2}{9} \quad \text{for} \quad k\geq 6,\\
		&&b''_3=-\frac{2}{3},\, b''_4=0,\, b''_5=-\frac{8}{15},\, b''_6=-\frac{22}{45},\,b''_K=\frac{8}{315}-\frac{4K}{45},\quad \text{for} \quad K\geq7,\\
		&&c''_3=\frac{2}{3},\,c''_4=0,\,c''_5=\frac{4}{5},\,c''_6=\frac{34}{45},\, c''_7=\frac{68}{63},\,c''_K=\frac{2(K+1)(126K+67)}{14175},\quad \text{for} \quad K\geq 8.
		\feq
		Combining these equations with \eqref{D00abc} and \eqref{ckey}, we conclude that
		\begin{align*}
		\sum_{K\geq 3} E((D_{1,K+1})^2)x^K &= \frac{\partial^2}{\partial u^2}D_1(x,0,0,u)\mid_{u=1}\\
		&=\frac{2}{3}x^3+\frac{4}{5}x^5+\frac{34}{45}x^6+\frac{68}{63}x^K+
		\sum_{K\geq8}\frac{2(K+1)(126K+67)}{14175}x^7.
		\end{align*}
		Hence, the variance of $D_{1,K+1}$ for $K\geq 8$ is given by
		$$\sigma^2(D_{1,K+1})=\frac{2(K+1)(126K+67)}{14175}+\frac{2(K+1)}{15}-\frac{4(K+1)^2}{15^2}
		=\frac{1772(K+1)}{14175},$$
		which completes the proof of part (ii) of the theorem.
	\end{proof}
	
	\subsection{Moments of $D_{i,K}:$ case $i\geq 2$}
	\label{ilarge}
Next, we focus on computing the first two moments of $D_{i,K}$ for $i\geq 2.$ Our method as presented in the previous subsection in finding moments of $D_{1,k}$ allows in principle to compute the moments recursively for any $i\in\nn.$ To illustrate the method, we will provide a detailed calculation for another case, namely $i=2$ and state the results for $3\leq i\leq 7$.
For the sake of simplicity, we assume $K\geq 31.$ Starting from $K=31$ the computation of moment is generic for all values, the computation for lower values of $K$ can be carried out in a similar way, but would be complicated by the necessity to consider multiple special cases.

	\begin{theorem}
		\label{thm:D1K1}
		Suppose that $K\geq 31.$ Then the  following holds true:
		 \begin{align*} &E(D_{2,K})=\frac{K}{9},&&E(D_{3,K})=\frac{2K}{35},&&E(D_{4,K})=\frac{K}{45},\\ &E(D_{5,K})=\frac{4K}{567},&&E(D_{6,K})=\frac{K}{525},&&E(D_{7,K})=\frac{2K}{4455},\\
			&\sigma^2(D_{2,K})=\frac{32K}{405},&&\sigma^2(D_{3,K})=\frac{119732K}{2837835},&&\sigma^2(D_{4,K})=\frac{12154K}{637875},\\
			&\sigma^2(D_{5,K}) =\frac{649555688 K}{97692469875},&&\sigma^2(D_{6,K})=\frac{5967328K}{3192564375},&&\sigma^2(D_{7,K})=\frac{191501338988K}{428772250281375}.
			\end{align*}
	\end{theorem}	
Similarly to Corollary~~\ref{cln} we have
\begin{corollary}
\label{cln1}
For $i\in [2,7],$ $\frac{D_{i,K}}{E(D_{i,K})}\overset{P}{\to}1$ as $K$ tends to infinity.
\end{corollary}  	
	\begin{proof}[Proof of Theorem~\ref{thm:D1K1}]
		 For integer $i\geq2,$ define \beq E_i(x,y,z)=\frac{\partial}{\partial u}D_i(x,y,z,u)\mid_{u=1} \quad \text{and} \quad  F_i(x,y,z)=\frac{\partial^2}{\partial u^2}D_i(x,y,z,u)\mid_{u=1}.
		\feq
		Let now $i=2$. Differentiating \eqref{DiPDE} with respect to $u$ and letting $u=1$, we obtain
		\begin{align*}
		&yz\frac{\partial}{\partial x} E_2(x,y,z)=\\
		&\quad z(E_2(x,y,z,u) - E_2(x,0,z,u))+y(E_2(x,y,z,u) - E_2(x,y,0,u))\\
		&\qquad +yz\left(E_2(x,y,0)+x(x+y)\right)\left(\frac{x^3}{(1-x)(1-z)}+x+x(x+z)+\frac{xz(z+x)}{1-z}\right)\\ &\qquad +yz\left(\frac{x^3}{(1-x)(1-y)}+x+x(x+y)+\frac{xy(y+x)}{1-y}\right)\left(E_2(x,0,z)+x(x+z)\right)
		\end{align*}
		and
		\begin{align*}
		yz\frac{\partial F_2(x,y,z)}{\partial x}
		&= z(F_2(x,y,z) - F_2(x,0,z))+y(F_2(x,y,z) - F_2(x,y,0))\\
		&\quad +  yzF_2(x,y,0)\left(\frac{x^3}{(1-x)(1-z)}+x+x(x+z) + \frac{xz(z+x)}{1-z}\right)\\
		&\quad +2yz \left(E_2(x,y,0)+x(x+y)\right)\left(E_2(x,0,z)+x(x+z)\right)\\
		&\quad +yz \left(\frac{x^3}{(1-x)(1-y)}+x+x(x+y) + \frac{xy(y+x)}{1-y}   \right)F_2(x,0,z),
		\end{align*}
		where we used the fact that $D_2(x,y,z,1)=\frac{x^3}{(1-x)(1-y)(1-z)}.$
		Leveraging any computational mathematics software such as Maple, one can verify that the solution of these equations are given by, respectively,
		\beq
		&&3(1-y)(1-z)(1-x)^2E_2(x,y,z) = -(y^2+z^2-y-z)x^3 +(y^2+z^2-2y-2z+2)x^4 \\
		&&\qquad \qquad\qquad\qquad\qquad\qquad\qquad\qquad+(y+z-2)x^5+\frac{1}{3}x^6.
		\feq
		and
		\begin{align*}
		&F_2(x,y,z)=
		\\
		&\quad 405(1-x)^3F_2(x,y,z) = 27yzx^3+270(-3yz+y+z)x^4+54(17yz-15y-15z+5)x^5\\
		&\qquad-\frac{18(28y^2z^2-79y^2z-79yz^2+56y^2+175yz+56z^2-101y-101z+45)}{(1-y)(1-z)}x^6\\
		&\qquad+\frac{18(8y^2z^2-36y^2z-36yz^2+38y^2+115yz+38z^2-94y-94z+61)}{(1-y)(1-z)}x^7\\
		&\qquad-\frac{18(y^2z^2-9y^2z-9yz^2+15y^2+45yz+15z^2-53y-53z+48)}{(1-y)(1-z)}x^8\\
		&\qquad-\frac{6(3y^2z+3yz^2-10y^2-30yz-10z^2+55y+55z-71)}{(1-y)(1-z)}x^9\\
		&\qquad-\frac{6(y^2+3yz+z^2-11y-11z+22)}{(1-y)(1-z)}x^{10}-\frac{6(y+z-4)}{(1-y)(1-z)}x^{11}
		-\frac{2}{(1-y)(1-z)}x^{12}.
		\end{align*}
		The generating functions $E_i(x,y,z)$ and $F_i(x,y,z)$ for all $i\geq 1$ can be in principle calculated in a similar fashion. We omit the details due to the length and complexity of expressions and only report the results for $E_i(x,0,0),$ and $F_i(x,0,0)$ (see Tables~\ref{tab:Ei} and~\ref{tab:Fi}).
		With this information in hand we are now in a position to calculate the mean and variance of $D_{i, K}.$ We accomplish this by virtue of
\eqref{D00}, and the fact that $\sum_{K=3}^\infty E(D_{i,K})x^K=E_i(x,0,0)$ and  $\sum_{K=3}^\infty E(D_{i,K}^2)x^K=F_i(x,0,0)+E_i(x,0,0).$	
	\end{proof}		
\subsection{CLT for $D_{i,K}$}
\label{clti}
	We conclude with a brief discussion of a central limit theorem for $D_{i,K}.$ Using a similar argument as in case of $D_1(x,y,z,u),$ one can show that the generating function $D_i(x,y,z,u)$ has the form
		\beq
		&&
		D_i(x,y,z,u)=
		\\
		&&
		\quad
		\sum_{K\geq 3} \frac{\sum_{0\leq a<b\leq i} A_{K,i,a,b}(u)(y^az^b+y^bz^a) +\sum_{a=0}^{i-1} B_{K,i,a}(u)(y^a+z^a) + C_{K,i}(u)}{(1-y)(1-z)}x^K,
		\feq
		and that the coefficients $A_{K,i,a,b}(u)$, $B_{K,i,a}(u)$ and $C_{K,i}(u)$ are polynomials.
		\par
		We believe that this information can be leveraged to prove the following result.
\begin{conj*}
\item [(i)] For every $i\in\nn,$ there exist $K_i\in\nn$ and positive $a_i,b_i\in\qq$ such that $E(D_{i,K})=a_iK$ and $\sigma(D_{i,K})=b_iK$  for all $K>K_i.$
\item [(ii)] For $i\geq 1,$ let $\witi D_{i,K} = \frac{D_{i,K} - E(D_{i,K})}{\sigma(D_{i,K})}.$ Then
$\lim_{K\to \infty} \call{(\witi D_{i,K})} = \caln(0,1),$
			where $\caln(0,1)$ is a standard normal distribution.
\end{conj*}
The above stated result is supported by our intensive numerical simulations (cf. Figure~\ref{fig:simul200k}). In fact, we believe that a stronger conjecture is true. Set $\witi D_{i,K}=0$ for an integer $i\geq K$ and use the notation $\witi D_K$ for the (infinite) vector $\bigl(\witi D_{i,K}\bigr)_{i\in\nn}.$ We have:
\begin{conj*}
As $K$ tends to infinity, $\witi D_K$ converges weakly to a Gaussian process in the product space $\rr^\nn.$
\end{conj*}

	\begin{table}[H]	
		\begin{adjustbox}{center}
	 		\begin{tabular}{ll}
				\toprule 	
				\multicolumn{1}{l}{$i$}
				& \multicolumn{1}{l}{$E_i(x,0,0)$} \\
				\midrule \normalsize
				$1$ & $15(1-x)^2E_1(x,0,0)=2(x^2-5x+5)$ \\
				$2$ & $9(1-x)^2E_2(x,0,0)=x^4(x^2-6x+6)$  \\
				$3$ & $105(1-x)^2E_3(x,0,0)=2x^3(3x^4-21x^3+56x^2-70x+35)$ \\ 		
				$4$ & $45(1-x)^2E_4(x,0,0)=x^4(x^2-3x+3)(x^2-5x+5)$ \\
				$5$ & $2835(1-x)^2E_5(x,0,0)=2x^5(10x^4-90x^3+279x^2-378x+189)$ \\
				$6$ & $1575(1-x)^2E_6(x,0,0)=x^6(3x^4-30x^3+100x^2-140x+70)$\\
				$7$ & $31185(1-x)^2E_7(x,0,0)=2x^7(7x^4-77x^3+275x^2-396x+198)$\\
				\bottomrule
			\end{tabular}

		\end{adjustbox}
		\caption{$E_i(x,0,0)$, for $i=1,2,3,4,5,6,7$ \label{tab:Ei}}
	\end{table}
	\begin{table}[H]
		\begin{adjustbox}{center}
			\begin{tabular}{ll}
				\toprule
				\multicolumn{1}{l}{$i$} & \multicolumn{1}{l}{$F_i(x,0,0)$} \\
				\midrule
				$1$ & $14175(1-x)^3F_1(x,0,0) =2x^3(4725-14175x+19845x^2-16380x^3$\\
				$ $ & $\qquad \qquad \qquad \qquad \qquad\quad  +8595x^4-2880x^5+580x^6-58x^7)$\\
				$2$ & $405(1-x)^3F_2(x,0,0)=2x^5(15-15x+6x^2-x^3)(3-3x+x^2)^2$\\
				$3$ & $14189175(1-x)^3F_3(x,0,0)=2x^7(1711710-5135130x+6786780x^2$\\
				$  $& $\qquad\qquad \qquad \qquad \qquad \qquad -5118113x^3+2380287x^4-682864x^5+111636x^6-7974x^7)$\\
				$4$ &$637875(1-x)^3F_4(x,0,0)=2x^9(15525-46575x+60255x^2-43695x^3$\\
				$ $  & $\qquad \qquad \qquad \qquad\qquad\qquad+19200 x^4-5100x^5+752x^6-47x^7)$\\
				$5$  &$97692469875(1-x)^4F_5(x,0,0)=4x^{11}(157260285-471780855x+598855005x^2$\\
				$ $&$\qquad \qquad\qquad\qquad\qquad\qquad\qquad -418392270x^3 +173906073x^4-42827760x^5$\\
				$ $ & $\qquad \qquad\qquad\qquad\qquad\qquad\qquad +5728788x^6-318266x^7)$\\
				$6$ & $3192564375(1-x)^3F_6(x,0,0)=2x^{13}(969150-2907450x+3627930x^2$\\
				$ $ & $\qquad \qquad\qquad\qquad\qquad\qquad\qquad -2446595x^3+963525x^4-220570x^5+26940x^6-1347x^7)$ \\
				$7$ &$428772250281375(1-x)^3F_7(x,0,0)=2x^{15}(9215899308-27647697924x$\\
				$ $ &$\qquad \qquad\qquad\qquad\qquad\qquad\qquad \qquad +33973625070x^2-22162777791x^3+8287091967x^4 $\\
				$ $&$\qquad \qquad\qquad\qquad\qquad\qquad\qquad \qquad -1769271504x^5+198572308x^6-9026014x^7)$ \\
				\bottomrule
			\end{tabular}
		\end{adjustbox}
		\caption{$F_i(x,0,0)$, for $i=1,2,3,4,5,6,7$ \label{tab:Fi}}
	\end{table}
\par

\section*{Acknowledgement}
We are grateful to the referee for comments and feedback on the earlier version of the manuscript that resulted in a better presentation of results and proofs.
	
\filbreak

	\nocite{*}
\end{document}